\documentclass{amsart}

\usepackage{amsmath,amsthm,amsfonts,amssymb}

\usepackage{mathtools} %extension of amsmath, provides bug fixes and other tools 
\usepackage{mathrsfs} %Provides a nice \mathscr command

\usepackage[active]{srcltx} %creates the .dvi file with links to the source .tex file
\usepackage[pdftex,bookmarks=true]{hyperref} %adds hypertext links, can't use in ``draft'' mode, and must use in pdf mode
\usepackage{graphicx}

 \usepackage{enumerate}
\usepackage{macros}
\usepackage{lipsum}
\usepackage{bbm}
\usepackage{lipsum}
\usepackage{amsfonts}
\usepackage{enumitem}
\usepackage{amsmath}
\usepackage{graphicx}
\usepackage{epstopdf}
\usepackage{algorithmic}
\ifpdf
  \DeclareGraphicsExtensions{.eps,.pdf,.png,.jpg}
\else
  \DeclareGraphicsExtensions{.eps}
\fi

\newtheorem{theorem}{Theorem}[section]
\newtheorem{lemma}[theorem]{Lemma}
\newtheorem{proposition}[theorem]{Proposition}

\newtheorem{definition}[theorem]{Definition}

\newtheorem{remark}[theorem]{Remark}

\numberwithin{equation}{section}

\title[OPTIMAL POINTWISE ERROR BOUNDS AND DIFFERENCE QUOTIENTS FOR POD]{On Optimal Pointwise in Time Error Bounds and
Difference Quotients for the Proper Orthogonal Decomposition}

\author[B. KOC]{BIRGUL KOC}
\address{Department of Mathematics, Virginia Tech, USA}
\email{birgul@vt.edu}
\urladdr{https://intranet.math.vt.edu/people/birgul/}
\thanks{Funding: The research of the first and fifth authors was supported by the National Science
Foundation under grant DMS-2012253. The research of the second author was supported by Spanish MCINYU under grant RTI2018-093521-B-C31 and Spanish State Research Agency through the
national programme Juan de la Cierva-Incorporaci\'{o}n 2017. This work has also been supported by
European Union’s Horizon 2020 research and innovation programme under the Marie Sklodowska–
Curie Actions, grant agreement 872442 (ARIA). The work of the third author was supported by the
National Science Foundation under grant DMS-1439786 and by the Simons Foundation under grant
50736 while the third author was in residence at the Institute for Computational and Experimental
Research in Mathematics in Providence, RI, during the “Model and dimension reduction in uncertain
and dynamic systems” program.}

\author[S. RUBINO]{SAMUELE RUBINO}
\address{Departamento EDAN \& IMUS, Universidad de Sevilla, Spain}
\email{samuele@us.es}
\urladdr{https://www.imus.us.es/en/fichapersonal/samuele}
\thanks{}

\author[M. SCHNEIER]{MICHAEL SCHNEIER}
\address{Department of Mathematics, University of Pittsburgh, USA}
\email{mhs64@pitt.edu}
\urladdr{https://www.mathematics.pitt.edu/people/michael-schneier}
\thanks{}

\author[J. R. SINGLER]{JOHN R. SINGLER}
\address{Department of Mathematics and Statistics, Missouri University of Science and Technology, USA}
\email{singlerj@mst.edu}
\urladdr{https://people.mst.edu/faculty/singlerj/index.html}
\thanks{}

\author[T. ILIESCU]{TRAIAN ILIESCU}
\address{Department of Mathematics, Virginia Tech, USA}
\email{iliescu@vt.edu}
\urladdr{https://intranet.math.vt.edu/people/iliescu/}
\thanks{}



\date{}

\dedicatory{}

\begin{document}

\maketitle

\begin{abstract}
In this paper, we resolve several long standing issues dealing with optimal pointwise in time error bounds for proper orthogonal decomposition (POD) reduced order modeling 
%model order reduction 
of the heat equation.
In particular, we study the role played by difference quotients (DQs) in obtaining reduced order model (ROM) error bounds that are optimal with respect to both the time discretization error and the ROM discretization error.
When the DQs are not used, we prove that both the ROM projection error and the ROM error are suboptimal.
When the DQs are used, we prove that both the ROM projection error and the ROM error are optimal.
The numerical results for the heat equation support the theoretical results.
\end{abstract}

\section{Introduction}
%%%%%%%%%%%%%%%%%%%%%%%
%%------------------------------------------------------------------------------------------------
% \section{Notation and Preliminaries}
% %%------------------------------------------------------------------------------------------------

In this paper, we consider the one-dimensional heat equation
\begin{eqnarray}
    u_t 
    -\nu \, u _{xx} 
    = f  \, , 
%    \text{on } [0,1] \times[0,T] \
    \label{eqn:heat}
\end{eqnarray}
where the spatial domain is $[0,1]$, the time domain is $[0,T]$, and $\nu$ is the diffusion coefficient.
For simplicity, we consider homogeneous Dirichlet boundary conditions $u(0,t) = u(1,t) = 0$ for $ t > 0 $ and given initial conditions $u(x,0) = u_0(x)$.
%We emphasize, however, that our discussion  carries over in a straightforward way to general parabolic equations.
We emphasize that, although our theoretical developments and numerical illustrations are for the heat equation, we believe that our analysis can be extended to more general parabolic equations, e.g., the Navier-Stokes equations.

We also consider projection reduced order models (ROMs) for the heat equation.
Specifically, we consider the proper orthogonal decomposition (POD)~\cite{HLB96}, which can be summarized as follows:
(i) The full order model (FOM) for~\eqref{eqn:heat} is run for selected parameter values and/or time intervals to generate a set of snapshots $\{ u^{0}, u^{1}, \ldots, u^{N} \}$;
(ii) These snapshots and the singular value decomposition (SVD) are used to construct an orthonormal ROM basis $\{ \varphi_{1}, \ldots, \varphi_{s} \}$ for a Hilbert space $ \mathcal{H} $, where $s$ is the rank of the snapshot matrix;
(iii) The ROM approximation
\begin{eqnarray}
    u(x,t_n)
    \approx u_{r}^{n}(x)
    = \sum_{j=1}^{r} u_{j}^{n} \, \varphi_{j}(x) \, ,
    \qquad
    n = 1, \ldots, N \, ,
    \label{eqn:introduction-1}
\end{eqnarray}
where $r < s$ is the ROM dimension, is used together with a Galerkin projection and a time discretization to yield a system of equations for $u_{j}^{n}$, which are the sought ROM coefficients.

%\red{
\begin{definition}[Generic Constant $C$]
For clarity, in what follows, we will denote by $C$ a generic positive constant that may vary from a line to another, but which is always independent of the discretization parameters. 
%(e.g., $r$ and $\Delta t$).
	\label{def:constant}
\end{definition}
%}

In the pioneering paper~\cite{KV01}, Kunisch and Volkwein laid the foundations of numerical analysis for POD 
%\red{
(see, e.g.,~\cite{KV02,luo2008mixed,rathinam2003new} for relevant work).
%}
In particular, for the ROM error 
\begin{eqnarray}
    e^{n}(x)
    = u(x , t_n) - u_{r}^{n} (x) \, ,
    \qquad
    n = 1, \ldots, N \, ,
    \label{eqn:introduction-1b}
\end{eqnarray}
they proved the following error bound (see Theorem 7 in~\cite{KV01}):
%\blue{Do they prove pointwise error bounds?  Theorem 7 is for averages.} \red{Michael: They do not.}
%$\forall \, n = 0, 1, \ldots, N$, 
\begin{align}
      \hspace*{0.25cm}
   \frac{1}{N+1} \sum_{n=1}^{N}\| e^{n} \|_{L^2}^2
    \leq C \, \biggl( 
    \text{time discretization error}
    + \text{ROM discretization error}
    \biggr) .
    \label{eqn:introduction-2}
\end{align}
% \purple{
% \begin{eqnarray}
%     \frac{1}{N+1}\sum_{n=1}^{N}\| e^{n} \|_{L^2}^2
%     \leq C \, \biggl( 
%         && \text{time discretization error}
%     \nonumber \\
%     && + \, \text{ROM discretization error}
%     \biggr) \, .
%     \label{eqn:introduction-2}
% \end{eqnarray}
% }
%where $C$ is a generic constant that does not depend on the discretization parameters.
% \blue{Define $C$ formally?} \red{Michael: 
% I would leave it as is, it's more readable this way.}
This estimate was later extended to include the spatial discretization error and a pointwise in time estimate in~\cite{iliescu2014variational},
%\blue{[I removed \cite{iliescu2013variational} since pointwise error is not considered there]},
%\red{
(see, e.g.,~\cite{kostova2018model,rathinam2003new} for alternative pointwise in time estimates)
%}
i.e., 
\begin{eqnarray}
    \| e^{n} \|_{L^2}
    \leq C \, \biggl( 
    && \text{space discretization error}
    + \text{time discretization error}
    \nonumber \\
    &+& \, \text{ROM discretization error}
    \biggr) \, .
    \label{eqn:introduction-temp}
\end{eqnarray}
%

% A straightforward extension of the error bound~\eqref{eqn:introduction-2} that includes the spatial discretization error can be made (see, e.g.,~\cite{iliescu2013variational,iliescu2014variational}):
% \begin{eqnarray}
%     \frac{1}{N+1}\sum_{n=1}^{N}\| e^{n} \|_{L^2}
%     \leq C \, \biggl( 
%     && \text{space discretization error}
%     + \text{time discretization error}
%     \nonumber \\
%     && + \, \text{ROM discretization error}
%     \biggr) \, .
%     \label{eqn:introduction-3}
% \end{eqnarray}
% \red{These time averaged estimates were then extended to the pointwise case in~\cite{iliescu2013variational}, i.e.
% \begin{eqnarray}
%     \| e^{n} \|_{L^2}
%     \leq C \, \biggl( 
%     && \text{space discretization error}
%     + \text{time discretization error}
%     \nonumber \\
%     && + \, \text{ROM discretization error}
%     \biggr) \, .
%     \label{eqn:introduction-temp}
% \end{eqnarray}
%}
%\blue{
Estimate \eqref{eqn:introduction-temp} relied on an assumption about the POD projection error, which roughly says that the POD projection error at each time step is of the same order as the POD projection error at the remaining time steps.  This assumption has since been generally used in proving pointwise in time error bounds for parabolic equations.
%}

We emphasize that the error bound~\eqref{eqn:introduction-temp} includes all three ROM error sources:
(i) the space discretization error, which results from the spatial discretization of the heat equation~\eqref{eqn:heat} with classical numerical methods, e.g., finite elements (FEs);
(ii) the time discretization error, which results from the time discretization of the heat equation~\eqref{eqn:heat} with classical numerical methods, e.g., Euler or Crank-Nicolson methods; and
(iii) the ROM discretization error, which results from the truncation in~\eqref{eqn:introduction-1}.  
%\blue{
%We note that error bounds of the form \eqref{eqn:introduction-temp} and other types of POD ROM error bounds have been proved for a variety of PDEs and many standard and modified POD ROMs; see, e.g., 
%\blue{[references].}

%\smallskip

A fundamental issue in the POD numerical analysis is the {\it optimality} of the error bound~\eqref{eqn:introduction-temp}.
%~\eqref{eqn:counterexample_exact_error3}.
We emphasize that there are three types of optimality, corresponding to the three types of discretization levels:
(i) space discretization optimality;
(ii) time discretization optimality; and 
(iii) ROM discretization optimality.
We discuss each optimality type below:

\paragraph{\it{Space Discretization Optimality}}
For simplicity, we consider a FE spatial discretization.
We emphasize, however, that other standard numerical methods (e.g., finite difference,  spectral, or spectral element methods) could be considered.
An error bound is optimal with respect to the spatial discretization if the error scalings with respect to the spatial discretization parameters only are of the following form:
\begin{eqnarray}
    && \| e^{n} \|_{L^2}
    = \mathcal{O} (h^{m+1}) \, ,
    \label{eqn:introduction-4} \\[0.3cm]
    && \| \nabla e^{n} \|_{L^2}
    = \mathcal{O} (h^{m}) \, ,
    \label{eqn:introduction-5}
\end{eqnarray}
where $h$ is the size of the FE mesh and $m$ is the FE order.
Proving estimates that are optimal with respect to the spatial discretization is relatively straightforward (see, e.g.,~\cite{giere2015supg,iliescu2013variational,iliescu2014variational}), since it follows the standard FE numerical analysis~\cite{thomee2006galerkin}.
%(See, however,~\cite{giere2015supg} for the delicate interplay between FE and ROM numerical analysis that can yield different parameter scalings for numerical stabilizations.)
%This is probably the reason why 
Thus, the spatial discretization error component is generally ignored in POD numerical analysis papers (see, e.g.,~\cite{KV01}).
%\blue{Check POD with weighted inner product in~\cite{volkwein2013proper}.}
To simplify the presentation, we will not discuss the spatial discretization optimality in this paper.
We note, however, that our results can be extended in a straightforward mannner to include the spatial discretization optimality.

\paragraph{\it{Time Discretization Optimality}}
An error bound is optimal with respect to the time discretization if the error scalings with respect to the time discretization parameters only are of the following form:
\begin{eqnarray}
    && \| e^{n} \|_{L^2}
    = \mathcal{O} (\Delta t^{k}) \, ,
    \label{eqn:introduction-6} 
\end{eqnarray}
where $\Delta t$ is the time step size used in the time discretization, and $k$ is the time discretization order (e.g., $k=1$ for Euler's method, and $k=2$ for Crank-Nicolson).

%\medskip

The importance of the time discretization optimality was recognized early on.
In Remark 1 of~\cite{KV01}, Kunisch and Volkwein proposed the {\it difference quotients (DQs)}  (i.e., scaled snapshots of the form $(u^{n} - u^{n-1}) / \Delta t, \, n = 1, \ldots, N$) as a means to achieve time discretization optimality.
Specifically, on page 121 of~\cite{KV01}, the authors noted that, in the DQ case (i.e., if the DQs are used to build the POD basis), time discretization optimal error bounds of the type~\eqref{eqn:introduction-6} follow.
However, in the noDQ case (i.e., if the DQs are not used), %the authors noted that 
the error bound has a suboptimal $(\Delta t^{-1})$ factor. 
%in the error bound.

%\medskip

A major development in the study of POD optimality was made by Chapelle, Gariah, and Sainte-Marie in~\cite{chapelle2012galerkin}.
The authors showed that using the $L^2$ projection instead of the Ritz projection used in~\cite{KV01} (which is standard in the FE numerical analysis~\cite{thomee2006galerkin,wheeler1973priori}) 
avoids the difficulties posed by the POD  approximation of the time derivative, and eliminates the need to use DQs 
%\blue{
to achieve time discretization optimality. 
%}. 

\paragraph{\it{ROM Discretization Optimality}} The first discussion of the ROM discretization optimality was presented in~\cite{iliescu2014are}.  In that work, a pointwise in time error bound was said to be optimal with respect to the ROM discretization if the error scalings with respect to the ROM discretization parameters only take one of the following forms:
\begin{eqnarray}
    && \| e^{n} \|_{L^2}^{2}
    = \mathcal{O} \Bigg( \frac{1}{N+1} \, \sum_{n=0}^{N} \| \eta^{proj}(t_{n}) \|_{L^2}^{2} \Bigg)
    = \mathcal{O} \Bigg( \sum_{i=r+1}^{s} \lambda_i \Bigg) \, ,
    \label{eqn:introduction-7} \\[0.3cm]
    && \| \nabla e^{n} \|_{L^2}^{2}
    = \mathcal{O} \Bigg( \frac{1}{N+1} \, \sum_{n=0}^{N} \| \nabla \eta^{proj}(t_{n}) \|_{L^2}^{2} \Bigg) 
        = \mathcal{O} \Bigg( \sum_{i=r+1}^{s} \lambda_i \| \nabla \varphi_i \|_{L^2}^{2} \Bigg) \, ,
    \label{eqn:introduction-8}
\end{eqnarray}
where $\eta^{proj}$ is the {\it POD projection error}, which is defined as
\begin{eqnarray}
   \eta^{proj}(x,t)  
    =  u(x,t)- \sum_{i=1}^{r} \Big( u(\cdot,t), \varphi_i(\cdot) \Big)_{\mathcal{H}} \varphi_i(x) \, ,
    \label{eqn:projection-error} 
\end{eqnarray}
and $\lambda_i$ and $ \varphi_i $ are POD eigenvalues and modes.
% \medskip
The first significant development in the study of POD optimality was made in~\cite{iliescu2014are}, where it was shown that not using the DQs yields error bounds that may be optimal with respect to the time discretization 
%\blue{
(using the technique from~\cite{chapelle2012galerkin}), 
%}, 
but are suboptimal with respect to the ROM discretization.
Specifically, in the noDQ case, it was shown in~\cite{iliescu2014are} that 
\begin{align}
    \| e^{n} \|_{L^2}^{2}
    = \mathcal{O} \Bigg( \frac{1}{N+1} \, \sum_{n=0}^{N} \| \nabla \eta^{proj}(t_{n}) \|_{L^2}^{2} \Bigg) 
        = \mathcal{O} \Bigg( \sum_{i=r+1}^{s} \lambda_i \| \nabla \varphi_i \|_{L^2}^{2} \Bigg) \, ,
    \label{eqn:introduction-9}
\end{align}
which is suboptimal with respect to the ROM discretization.
Furthermore, in the DQ case, it was shown~\cite{iliescu2014are} that 
\begin{eqnarray}
    \| e^{n} \|_{L^2}^{2}
    = \mathcal{O} \Bigg( \frac{1}{N+1} \, \sum_{n=0}^{N} \| \eta^{proj}(t_{n}) \|_{L^2}^{2} \Bigg) 
        = \mathcal{O} \Bigg( \sum_{i=r+1}^{s} \lambda_i  \Bigg) \, ,
    \label{eqn:introduction-10}
\end{eqnarray}
which is optimal with respect to the ROM discretization.  However, two  assumptions on the POD projection errors were made in order to establish these results.

%\bigskip

To summarize, the current state-of-the-art in POD optimality {\it suggests} that 
\begin{eqnarray}
    \boxed{\text{DQs are needed for optimal POD  error bounds.}}
    \label{eqn:conjecture}
\end{eqnarray}

We emphasize that, to our knowledge, \eqref{eqn:conjecture} {\it has never been proved}.
%\smallskip
Indeed, \cite{KV01} focused on the time discretization optimality, but ignored the ROM discretization optimality.
Specifically, the authors proved that using DQs yields error bounds that are optimal with respect to the time discretization, but not necessarily with respect to the ROM discretization.
In~\cite{chapelle2012galerkin}, the authors considered the noDQ case and developed a framework that yields error bounds
%proved error bounds 
that are optimal with respect to the time discretization, but not necessarily with respect to the ROM discretization.
%\smallskip
A completely different approach was taken in~\cite{iliescu2014are}, where the focus was on ROM discretization optimality, without considering the time discretization optimality.
Specifically, in~\cite{iliescu2014are} it was shown both theoretically and numerically that, in the noDQ case the error bounds are suboptimal with respect to the ROM discretization error, whereas in the DQ case the error bounds are optimal.
The time discretization optimality was ignored in~\cite{iliescu2014are}.

%\smallskip

In this paper, we prove~\eqref{eqn:conjecture}. Specifically, we make three main contributions:

%\smallskip

First, in the noDQ case, we prove that the POD error bound is suboptimal not only with respect to the ROM discretization (as shown in~\cite{iliescu2014are}), but also with respect to the time discretization.
Specifically, we show that the scaling of the error bound~\eqref{eqn:introduction-9} with respect to the ROM discretization can degrade to
\begin{eqnarray}
    \| e^{n} \|_{L^2}^{2}
    = \mathcal{O} \Bigg( \Delta t ^{-1} \, \sum_{i=r+1}^{s} \lambda_i \Bigg)
    + \mathcal{O} \Bigg( \sum_{i=r+1}^{s} \lambda_i \| \nabla \varphi_i \|_{L^2}^{2} \Bigg) \, .
    \label{eqn:introduction-11}
\end{eqnarray}
In particular, we construct two analytical examples, and we prove that they satisfy~\eqref{eqn:introduction-11} in the noDQ case.
We note that the bound~\eqref{eqn:introduction-11} is a significant improvement over the bound~\eqref{eqn:introduction-9} proved in~\cite{iliescu2014are}, since the latter did not display the time discretization suboptimality.

%\smallskip

Our second main contribution is that we prove new pointwise in time error bounds in the DQ case, and we do not require any of the assumptions used in \cite{iliescu2014are} to establish similar pointwise bounds.  All of these error bounds are optimal with respect to the time discretization.  
%\blue{
One key component of our analysis is that we prove that an assumption from \cite{iliescu2014are,iliescu2014variational} concerning pointwise in time behavior of POD projection errors is automatically satisfied in the DQ case.
%}

%\smallskip

Our third main contribution is that we revisit the definition of ROM discretization error optimality, introduce a new stronger notion of optimality, and show that all of the pointwise in time error bounds in the DQ case are optimal in at least one sense.  Both pointwise in time error bounds using the $ H^1_0 $ norm are optimal in the new stronger sense; the pointwise in time bounds using the $ L^2 $ norm can be optimal in either sense.  We note that to prove the stronger optimality of the $ L^2 $ error bounds, we do need a uniform boundedness assumption of the type made in \cite{iliescu2014are}.

%\smallskip

We emphasize that, although our theoretical developments and numerical illustrations are for the heat equation, we believe that our analysis can be extended to more general parabolic equations, e.g., the Navier-Stokes equations.

%(Again, we do not prove spatial discretization optimality, but that is straightforward.)
%Specifically, for the Crank-Nicolson time discretization, we prove the following optimal error bound: 
%\textcolor{purple}{in 1.16, since we have square of error, should we have $\Delta t^4$ instead of $\Delta t^2$?}
%\begin{eqnarray}
%    \| e^{n} \|_{L^2}^{2}
%    = \mathcal{O} \Bigg( \sum_{i=r+1}^{R} \lambda_i \Bigg)
%    + \mathcal{O} \Bigg( \Delta t^{4} \Bigg) \, .
%    \label{eqn:introduction-12}
%\end{eqnarray}
%\medskip
% For the first two main contributions, we illustrate numerically the theoretical results.
% Specifically, for the heat equation~\eqref{eqn:heat} and both analytical examples, we show the following:
% (i) in the noDQ case, the error scales as in~\eqref{eqn:introduction-11} (i.e., is suboptimal), and
% (ii) in the DQ case, the error scales according to the new error bounds.

\smallskip

%\red{
%\begin{remark}[DQs and EIM/DEIM]
\paragraph{\it{DQs in Applications}}
The focus of this paper is on the role played by DQs in the POD numerical analysis.
%The DQs were introduced in~\cite{KV01} and used in, e.g.,~\cite{iliescu2014are,singler2014new} \blue{more refs!} for theoretical purposes (i.e., to prove optimal error bounds with respect to the time discretization).
We emphasize, however, that DQs are also widely used in practical applications.

Probably the most important use of DQs in practical computations is in {\it hyperreduction} methods for ROMs of nonlinear systems of the form $y'=f(t,y)$.
Hyperreduction methods~\cite{yano2019discontinuous} significantly decrease the computational cost of the nonlinear ROM operator evaluations, which can be prohibitive in realistic applications.
Popular hyperreduction methods (e.g., the empirical interpolation method (EIM)~\cite{barrault2004eim} and its discrete counterpart, the discrete empirical interpolation method (DEIM)~\cite{chaturantabut2012state})  use the nonlinear snapshots $f(t,y)$ to construct accurate approximations of the nonlinear ROM operators. 
As noted on page 48 in~\cite{chaturantabut2012state}, since $f(t,y) = y'$ and $(y^{n+1} - y^{n}) / \Delta t \approx y'$,
using nonlinear snapshots is similar to including the DQs.
The DQs' connection to nonlinear snapshots was also used in~\cite{choi2020sns} to develop the  solution-based nonlinear subspace (SNS) method as an efficient alternative to classical hyperreduction techniques.
The SNS method was used in the reduced order modeling of the nonlinear diffusion equation and the parameterized quasi-1D Euler equation.

The DQs were explicitly used in various practical applications. 
For example, the DQs were utilized to develop data-driven ROMs for turbulent flows, in which the eddy viscosity field is a function of the time history of the velocity field (see Section 3.3 in~\cite{hijazi2020data}).
The DQs were also used in the reduced order modeling of the FitzHugh–Nagumo equations, which are used to model the dynamics of a spiking neuron (see Section 4 in~\cite{kostova2018model}). 
Furthermore, the DQs were employed to construct ROMs for the control of laser surface hardening~\cite{homberg2003control},
%} 
%\blue{
for feedback control of various PDEs~\cite{LeibfritzVolkwein07}, for partial integro-differential equations arising in financial applications~\cite{sachs2013priori}, for subdiffusion equations~\cite{JinZhou17}, for convection-diffusion equations~\cite{ZhuDedeQuarteroni17},  for wave equations~\cite{herkt2013convergence,ZhuDedeQuarteroni17}, and for flow between offset cylinders and lid driven cavity flows~\cite{KeanSchneier20}.
%}

%\red{
In this paper, we use DQs with respect to time to obtain optimal pointwise in time error estimates.
A different, yet related approach was utilized in, e.g.,~\cite{carlberg2011low,ito1998reduced,zimmermann2013gradient}, where DQs with respect to system parameters and initial conditions were used to improve the predictive capabilities of reduced basis methods (RBMs)~\cite{hesthaven2015certified,quarteroni2015reduced} for parameterized problems.
In this setting, the noDQ case is referred to as the Lagrange approach, whereas the DQ case is referred to as the Hermite approach~\cite{ito1998reduced}.
The error sensitivity with respect to parameters was investigated in, e.g.,~\cite{homescu2005error,rathinam2003new}.
%}

\smallskip

The rest of the paper is organized as follows.
In Section~\ref{sec:POD}, we describe the POD construction in the noDQ and DQ cases.
%We establish the main contributions by first thoroughly investigating the previously described POD pointwise projection error assumption. 
In Section \ref{sec:pointwise-projection-error-estimates}, we give 
%more background and 
more detail about the previously described POD pointwise projection error assumption,
%this assumption, 
show using examples that it can fail in the noDQ case, and prove that it is always satisfied in the DQ case.  These results allow us to complete the POD ROM error analysis in Section \ref{sec:ErrEst}.
For the first two main contributions, in Section~\ref{sec:numerical-results} we illustrate numerically the theoretical results.
Specifically, for the heat equation~\eqref{eqn:heat} and both analytical examples, we show the following:
(i) in the noDQ case, the error scales as in~\eqref{eqn:introduction-11} (i.e., is suboptimal), and
(ii) in the DQ case, the error scales according to the new error bounds.
Finally, in Section~\ref{sec:conclusions}, we present our conclusions and future research directions.

\section{Proper Orthogonal Decomposition (POD)}
\label{sec:POD}
In this section we introduce two different approaches for constructing our reduced basis  by using the {\it proper orthogonal decomposition (POD)}~\cite{HLB96,volkwein2013proper}.  
Suppose we have a collection of snapshots $ U = \{ u^n \}_{n=0}^N $ contained in a real Hilbert space $ \mathcal{H} $.  We consider the situation where each snapshot $ u^n $ is equal to $ u(t_n) $, where $ u \in C([0,T];\mathcal{H}) $ and $ t_n = n \Delta t $ for $ n = 0, \ldots, N $ so that $ t_0 = 0 $, $ t_N = T $, and $ \Delta t = T/N $.  We assume $ T > 0 $ is a fixed final time; however, $ \Delta t $ and $ N $ are allowed to vary.

%%------------------------------------------------------------------------------------------------
\subsection{POD Without Difference Quotients (noDQ Case)}
\label{subsec:POD_no_DQs}
%%------------------------------------------------------------------------------------------------

We begin by examining the POD problem without difference quotients. 
In what follows, we denote this case the {\it noDQ case}. 
Given a fixed $ r > 0 $, the problem is to find a set of orthonormal basis functions $ \{ \varphi_i \}_{i=1}^r \subset \mathcal{H} $, called POD modes or POD basis functions, that optimally approximate the snapshots in the sense that the following error measure is minimized:
\begin{equation}\label{eqn:POD_approx_error_def}
  E_r = \frac{1}{N+1} \sum_{n=0}^N \left\| u^n - P_r u^n \right\|_\mathcal{H}^2,
\end{equation}
 where $ P_r : \mathcal{H} \to \mathcal{H}$ is the orthogonal projection onto $ X^r = \mathrm{span}\{ \varphi_i \}_{i=1}^r $ given by
\begin{equation}\label{eqn:POD_projection_P_r}
  P_r u = \sum_{i=1}^r ( u, \varphi_i )_\mathcal{H} \varphi_i,  \quad  u \in \mathcal{H}.
\end{equation}

One way to find a solution of this problem is to solve the eigenvalue problem
\begin{equation}\label{eqn:POD_eigenvalue_problem}
  K \mathbf{z}_i = \lambda_i \mathbf{z}_i,  \quad  \mbox{for $ i = 1, \ldots, r $,}
\end{equation}
where $ K $ is the snapshot correlation matrix with entries
\begin{equation}\label{eqn:POD_corr_matrix}
  K_{mn} = \frac{1}{N+1} \, ( u^m, u^n )_\mathcal{H},  \quad  m,n = 0, \ldots, N.
\end{equation}
We order the eigenvalues $ \{ \lambda_i \} $ and corresponding orthonormal eigenvectors $ \{ \mathbf{z}_i \} $ so that $ \lambda_1 \geq \lambda_2 \geq \lambda_{N+1} \geq 0 $.  The optimizing orthonormal set $ \{ \varphi_i \}_{i=1}^r \subset \mathcal{H} $ is given by
\begin{equation}\label{eqn:POD_basis}
  \varphi_i = \lambda_i^{-1/2} (N+1)^{-1/2} \sum_{m=0}^N ( \mathbf{z}_i )^m u^m,  \quad  i = 1, \ldots, r,
\end{equation}
where $ ( \mathbf{z}_i )^m $ is the $ m $th entry of $ \mathbf{z}_i $.  Using these POD modes gives the optimal value for the approximation error:
\begin{equation}\label{eqn:POD_optimal_error}
  \frac{1}{N+1} \sum_{n=0}^N \left\| u^n - P_r u^n \right\|_\mathcal{H}^2 = \sum_{i > r} \lambda_i.
\end{equation}

We note that the scaling factor $ (N+1)^{-1} $ is important if one is interested in the solution of the optimization problem as more snapshots are collected, i.e., as $ \Delta t $ decreases or $ N $ increases.  For certain choices of the scaling factor, the error measure $ E_r $ in \eqref{eqn:POD_approx_error_def} converges to a time integral or a constant multiple of a time integral, and the POD eigenvalues and POD modes also converge; see, e.g., \cite{galan2008error,GubischVolkwein17,KV02,Singler11} for more information.

Different choices for the scaling factor in~\eqref{eqn:POD_approx_error_def} have been used in the literature.  We fix the scaling factor throughout this work to be $ (N+1)^{-1} $ for simplicity.  We note that since $ \Delta t = T/N $, we have $ (N+1)^{-1} = T_1^{-1} \Delta t $, where $ T_1 = T + \Delta t $.  Therefore, $ E_r $ in \eqref{eqn:POD_approx_error_def} is equal to the left Riemann sum approximation of the integral
  $$
    \frac{1}{T_1} \int_0^{T_1}  \| u(t) - P_r u(t) \|^2_\mathcal{H} \, dt.
  $$
We note that the results in this work will hold for other scaling factors, as long as the scaling factor in question scales like a constant multiple of $ \Delta t $.

\begin{remark}
  One can also consider variable time steps and weights in the POD problem; we only consider a constant time step and single weight $ (N+1)^{-1} $ for simplicity.  Furthermore, one can use other quadrature rules, such as the midpoint rule or trapezoid rule, to obtain appropriate weights for the POD problem.
\end{remark}

In the following result, we give POD approximation errors in different norms and using other projections onto $ X^r $.  Similar results have been proved in multiple works (see, e.g., \cite{iliescu2014are,iliescu2014variational,LockeSingler20,
%giere2015supg,
ShenSinglerZhang19,singler2014new}),  %\textcolor{blue}{and [other references?]}), 
and our proof relies on techniques from these works.  We note that this result can be obtained directly from the general results in the recent reference \cite{LockeSingler20}; however, we include a proof to be complete.  In this work, a bounded linear operator $ \Pi : Z \to Z $ for a normed space $ Z $ is a projection onto $ Z^r \subset Z $ if $ \Pi^2 = \Pi $ and the range of $ \Pi $ equals $ Z^r $.  In this case, $ \Pi z = z $ for any $ z \in Z^r $.

\begin{lemma}\label{lemma:POD_approx_errors_Wnorm}
  Let $ X^r = \mathrm{span}\{ \varphi_i \}_{i=1}^r \subset \mathcal{H} $, let $ P_r : \mathcal{H} \to \mathcal{H}$ be the orthogonal projection onto $ X^r $ as defined in \eqref{eqn:POD_projection_P_r}, and let $ s $ be the number of positive POD eigenvalues for $ U = \{ u^n \}_{n=0}^N $.  If $ W $ is a real Hilbert space with $ U \subset W $ and $ R_r : W \to W $ is a bounded linear projection onto $ X^r $, then
  \begin{align}
    \frac{1}{N+1} \sum_{n=0}^N \left\| u^n - P_r u^n \right\|_W^2 &= \sum_{i = r+1}^s \lambda_i \| \varphi_i \|_W^2,\label{eqn:POD_proj_error_W_norm}\\
    \frac{1}{N+1} \sum_{n=0}^N \left\| u^n - R_r u^n \right\|_W^2 &= \sum_{i = r+1}^s \lambda_i \| \varphi_i - R_r \varphi_i \|_W^2.\label{eqn:POD_proj_error_W_norm2}    
  \end{align}
\end{lemma}
\begin{proof}
  First, we note that \eqref{eqn:POD_proj_error_W_norm} is a special case of \eqref{eqn:POD_proj_error_W_norm2} since $ P_r \varphi_i = 0 $ for $ i > r $.  Therefore, we only prove \eqref{eqn:POD_proj_error_W_norm2}.
  
  Next, by the POD approximation error formula \eqref{eqn:POD_optimal_error}, we have $ u^n = P_s u^n $ for each $ n $.  If $ r \geq s $, since $ R_r $ is a projection onto $ X^r $ we have $ R_r u^n = R_r P_s u^n = P_s u^n = u^n $ and this proves the result.  Therefore, assume $ r < s $.  Note by the definition of $ \varphi_i $ in \eqref{eqn:POD_basis}, since $ u^n \in W $ for each $ n $ we have $ \varphi_i \in W $ for $ i = 1, \ldots, r $.  Therefore, $ X^r \subset W $, and since the range of $ R_r $ equals $ X^r $ we know the $ W $ norm in \eqref{eqn:POD_proj_error_W_norm2} is well-defined.
  
  Now, using the definition of $ P_r $ in \eqref{eqn:POD_projection_P_r} gives
  %
%   \begin{align*}
%   %
%     \frac{1}{N+1} \sum_{n=0}^N \left\| u^n - R_r u^n \right\|_W^2  &=  \frac{1}{N+1} \sum_{n=0}^N ( ( I - R_r ) P_s u^n, ( I - R_r ) P_s u^n )_W\\
%     %
%       &=  \frac{1}{N+1} \sum_{n=0}^N \sum_{i,j=1}^s ( u^n, \varphi_j)_\mathcal{H} ( u^n, \varphi_i)_\mathcal{H} ( ( I - R_r ) \varphi_j, ( I - R_r ) \varphi_i )_W,
%   %
%   \end{align*}
  %
% \purple{
   \begin{align*}
    \frac{1}{N+1} & \sum_{n=0}^N \left\| u^n - R_r u^n \right\|_W^2  =  \frac{1}{N+1} \sum_{n=0}^N ( ( I - R_r ) P_s u^n, ( I - R_r ) P_s u^n )_W\\
      &=  \frac{1}{N+1} \sum_{n=0}^N \sum_{i,j=1}^s ( u^n, \varphi_j)_\mathcal{H} ( u^n, \varphi_i)_\mathcal{H} ( ( I - R_r ) \varphi_j, ( I - R_r ) \varphi_i )_W,
  \end{align*}
 % }
  where $ I $ is the identity operator.  Next, take the $ \mathcal{H} $ inner product of \eqref{eqn:POD_basis} with $ u^n $ and use the eigenvalue equations \eqref{eqn:POD_eigenvalue_problem}-\eqref{eqn:POD_corr_matrix} to get
  $$
    ( u^n, \varphi_i )_\mathcal{H} = (N+1)^{1/2} \lambda_i^{1/2} ( \mathbf{z}_i )^n.
  $$
  Using this and also that $ \{ \mathbf{z}_i \} $ is orthonormal so that $ \sum_{n=0}^N ( \mathbf{z}_j )^n ( \mathbf{z}_i )^n = \delta_{ij} $ gives
  \begin{align*}
	\frac{1}{N+1} \sum_{n=0}^N \left\| u^n - R_r u^n \right\|_W^2  &=  \sum_{i,j=1}^s (\lambda_i \lambda_j)^{1/2} \delta_{ij} ( ( I - R_r ) \varphi_j, ( I - R_r ) \varphi_i )_W\\
	  &= \sum_{i=1}^s \lambda_i \| ( I - R_r ) \varphi_i \|_W^2.
  \end{align*}
  Since $ \varphi_i \in X^r $ for $ i = 1, \ldots, r $ and $ R_r $ is a projection onto $ X^r $, we have $ R_r \varphi_i = \varphi_i $ for $ i = 1, \ldots, r $ and this proves the result.
\end{proof}

%%------------------------------------------------------------------------------------------------
\subsection{POD With Difference Quotients (DQ Case)}
\label{sec:POD_diffQ}
%%------------------------------------------------------------------------------------------------

In this section we consider a POD problem for the same snapshots as those in Section~\ref{subsec:POD_no_DQs}, this time utilizing the difference quotients~\cite{KV01}: find an orthonormal set of basis functions $ \{ \varphi_i \}_{i=1}^r \subset \mathcal{H}$ minimizing the approximation error
\begin{equation}\label{eqn:POD_approx_error_def_with_DQs}
  E^\mathrm{DQ}_r = \frac{1}{2N+1} \sum_{n=0}^N \left\| u^n - P_r u^n \right\|_\mathcal{H}^2 + \frac{1}{2N+1} \sum_{n=1}^{N} \left\| \partial u^n - P_r \partial u^n \right\|_\mathcal{H}^2,
\end{equation}
where the {\it difference quotients (DQs)} $ \{ \partial u^n \}_{i=1}^{N} $ are defined by
\begin{equation}\label{eqn:DQ_def}
  \partial u^n = \frac{ u^{n} - u^{n-1} }{ \Delta t }.
\end{equation}
In what follows, we denote this case the {\it DQ case}.

The solution to this problem can be found by setting $ v^n = u^n $ for $ n = 0, \ldots, N $ and $ v^{N+n} = \partial u^n $ for $ n = 1, \ldots, N $.  This yields a new collection of snapshots $ U^\mathrm{DQ} = \{ v^n \}_{n=0}^{M} $, where $ M = 2N $.  Proceeding as outlined in Section \ref{subsec:POD_no_DQs} using the new collection $ \{ v^n \}_{n=0}^{M} $ in place of $ \{ u^n \}_{n=0}^{N} $ gives the solution of this different POD problem.  We use $ \{ \lambda_i^\mathrm{DQ} \} $ to denote the POD eigenvalues for this POD problem; we use the same notation $ \{ \varphi_i \}_{i=1}^r $ for the POD basis functions.  The optimal approximation error is given by
\begin{equation}\label{eqn:POD_optimal_error_with_DQs}
  \frac{1}{2N+1} \sum_{n=0}^N \left\| u^n - P_r u^n \right\|_\mathcal{H}^2 + \frac{1}{2N+1} \sum_{n=1}^{N} \left\| \partial u^n - P_r \partial u^n \right\|_\mathcal{H}^2 = \sum_{i > r} \lambda_i^\mathrm{DQ}.
\end{equation}

Again, the choice of the scaling factor in the approximation error \eqref{eqn:POD_approx_error_def_with_DQs} is important if we consider the case where the amount of data increases, i.e., $ \Delta t $ decreases and $ N $ increases.  The DQs are used to approximate the time derivative of the data; therefore, for an appropriate choice of the scaling factor the approximation error in \eqref{eqn:POD_approx_error_def_with_DQs} contains approximations of time integrals involving both the data $ u(t) $ and also the time derivative of the data $ \partial_t u(t) $.  For the DQ case, we use $ (2N+1)^{-1} $ for the scaling factor throughout for simplicity.

As before, we give POD approximation errors in different norms and using other projections onto $ X^r $.
\begin{lemma}\label{lemma:POD_approx_errors_Wnorm_withDQs}
	Let $ X^r = \mathrm{span}\{ \varphi_i \}_{i=1}^r \subset \mathcal{H} $, let $ P_r : \mathcal{H} \to \mathcal{H}$ be the orthogonal projection onto $ X^r $ as defined in \eqref{eqn:POD_projection_P_r}, and let $ s $ be the number of positive POD eigenvalues for the collection $ U^\mathrm{DQ} = \{ v^n \}_{n=0}^{2N} $ described above.  If $ W $ is a real Hilbert space with $ U^\mathrm{DQ} \subset W $ and $ R_r : W \to W $ is a bounded linear projection onto $ X^r $, then
	%
% 	\begin{align}
% 	%
% 	\frac{1}{2N+1} \sum_{n=0}^N \left\| u^n - P_r u^n \right\|_W^2 + \frac{1}{2N+1} \sum_{n=1}^{N} \left\| \partial u^n - P_r \partial u^n \right\|_W^2 &= \sum_{i = r+1}^s \lambda_i^\mathrm{DQ} \| \varphi_i \|_W^2,\label{eqn:POD_proj_error_W_norm_withDQs}\\
% 	%
% 	\frac{1}{2N+1} \sum_{n=0}^N \left\| u^n - R_r u^n \right\|_W^2 + \frac{1}{2N+1} \sum_{n=1}^{N} \left\| \partial u^n - R_r \partial u^n \right\|_W^2 &= \sum_{i = r+1}^s \lambda_i^\mathrm{DQ} \| \varphi_i - R_r \varphi_i \|_W^2.    \label{eqn:POD_proj_error_W_norm2_withDQs}
% 	%
% 	\end{align}
	%
	%
%\purple{
	\begin{align}
	\frac{1}{2N+1} \Bigg( \sum_{n=0}^N \left\| u^n - P_r u^n \right\|_W^2 +  \sum_{n=1}^{N} \left\| \partial u^n - P_r \partial u^n \right\|_W^2 \Bigg) &= \sum_{i = r+1}^s \lambda_i^\mathrm{DQ} \| \varphi_i \|_W^2,\label{eqn:POD_proj_error_W_norm_withDQs}\\
	\frac{1}{2N+1} \Bigg( \sum_{n=0}^N \left\| u^n - R_r u^n \right\|_W^2 +  \sum_{n=1}^{N} \left\| \partial u^n - R_r \partial u^n \right\|_W^2 \Bigg) &= \sum_{i = r+1}^s \lambda_i^\mathrm{DQ} \| \varphi_i - R_r \varphi_i \|_W^2.    \label{eqn:POD_proj_error_W_norm2_withDQs}
	\end{align}
%	}
\end{lemma}
\begin{proof}
	Apply Lemma \ref{lemma:POD_approx_errors_Wnorm} to the new collection of snapshots $ \{ v^n \}_{n=0}^{M} $ described above.
\end{proof}

\begin{remark}
In this section, we considered the DQs defined by \eqref{eqn:DQ_def}. In practice the definition of the DQs will reflect the time discretization used to collect the snapshot data. 
%\blue{
For example, POD with 
%\blue{
central difference quotients is used for wave equations in \cite{herkt2013convergence,ZhuDedeQuarteroni17} and
%} 
fractional difference quotients are used for a subdiffusion problem in \cite{JinZhou17}.
%} 
It is possible that the results of this paper can be extended to these and other definitions of the DQs, such as those arising from the backward differentiation formulas (BDF2, BDF3, etc.).  We leave this to be considered elsewhere.
\end{remark}

% \red{
% \begin{remark}[DQs and EIM/DEIM]
% The DQs were introduced in~\cite{KV01} and used in, e.g.,~\cite{iliescu2014are,singler2014new} \blue{more refs!} for theoretical purposes (i.e., to prove optimal error bounds with respect to the time discretization).
% We emphasize, however, that DQs were also used for practical purposes.
% For example, the DQs play a central role in the empirical interpolation method (EIM)~\cite{barrault2004eim} and its discrete counterpart, the discrete empirical interpolation method (DEIM)~\cite{chaturantabut2012state}, which are hyper-reduction methods for ROMs of nonlinear systems of the form $y'=f(t,y)$.  
% Indeed, both EIM and DEIM use the nonlinear snapshots $f(t,y)$. 
% As noted on page 48 in~\cite{chaturantabut2012state}, since $f(t,y) = y'$ and $(y^{n+1} - y^{n}) / \Delta t \approx y'$,
% using nonlinear snapshots is similar to including the temporal DQs.

% %The connections between the nonlinear snapshots and the DQs were also used in, e.g.,~\cite{choi2020sns,kostova2018model}.
% The DQs and their connection to nonlinear snapshots were used in numerous practical applications, e.g., the development of data-driven ROMs for turbulent flows~\cite{hijazi2020data}, 
% \blue{Add more refs from the introductions in \cite{choi2020sns,kostova2018model}.}
%     \label{remark:dqs-eim/deim}
% \end{remark}
% }

%%------------------------------------------------------------------------------------------------

% \section{POD Pointwise Error Estimates}
\section{Pointwise Projection Error Estimates}
    \label{sec:pointwise-projection-error-estimates}
%%------------------------------------------------------------------------------------------------

%\textcolor{blue}{We need references in this paragraph; here are a few, but I am sure we can add more:} \cite{iliescu2014variational,iliescu2014are,eroglu2017modular,gunzburger2017ensemble,mohebujjaman2017energy,xie2018numerical,ErogluKayaRebholz19,zerfas2019continuous,decaria2020artificial,KeanSchneier20}

In the current literature on pointwise error bounds for the POD of parabolic problems 
%(e.g., the heat equation~\cite{}, the convection-diffusion equation~\cite{}, and the Navier-Stokes equations~\cite{}), 
several researchers make an assumption concerning the pointwise in time behavior of the POD projection errors~\cite{decaria2020artificial,eroglu2017modular,ErogluKayaRebholz19,gunzburger2017ensemble,iliescu2014are,iliescu2014variational,KeanSchneier20,mohebujjaman2017energy,xie2018numerical,zerfas2019continuous}.
Roughly, the assumption says that the POD projection error at any time is of the same order as the total POD projection errors considered in Section \ref{sec:POD}.
Next, we formalize this assumption in Assumption~\ref{ass:LinftyTime}, and then we discuss it for the noDQ case (Section~\ref{sec:pointwise-error-noDQ}) and the DQ case (Section~\ref{sec:pointwise-error-DQ}).

\medskip

We consider the POD of a collection of snapshots $ U := \{ u^n \}_{n=0}^N \subset \mathcal{H} $ and also $ U \subset W $, as in Section \ref{sec:POD}.  Recall, $ P_r : \mathcal{H} \to \mathcal{H} $ is the orthogonal projection onto the first $ r $ POD modes.  For either the noDQ case or the DQ case, the pointwise POD projection error assumption is given as follows:
\begin{assumption}\label{ass:LinftyTime}
There exists a constant $ C $, depending on $T = N \Delta t $ only, such that the POD projection error satisfies
\begin{align}\label{eq:AssNoDQs}
\left\| u^n - P_r u^n \right\|_{W}^{2}
\leq C\sum_{i=r+1}^{s}\lambda_{i} \|\varphi_i\|_{W}^2 \quad \text{for all $ r =1, \ldots, s $ and $ n=0,\ldots,N $}.
\end{align}
\end{assumption}
In Section \ref{sec:pointwise-error-noDQ}, we construct examples that show that this assumption can be violated in the noDQ case.  In Section \ref{sec:pointwise-error-DQ}, we show in Theorem \ref{thm:uniform_estimates} that this assumption is always satisfied in the DQ case.

\begin{remark}[Avoiding Assumption~\ref{ass:LinftyTime}]
	\label{rem:LinftyTime-similar-assumptions}
We notice that Assumption~\ref{ass:LinftyTime} would follow directly from the POD approximation properties \eqref{eqn:POD_proj_error_W_norm} (in the noDQ case) and \eqref{eqn:POD_proj_error_W_norm_withDQs} (in the DQ case) if we dropped the $1/(N+1)$ and $1/(2N+1)$ factors in the definitions \eqref{eqn:POD_approx_error_def} and \eqref{eqn:POD_approx_error_def_with_DQs} of the error measures $ E_r $ and $ E_r^\mathrm{DQ} $. 
%\blue{
In fact, when $\mathcal{H} = W=\mathbbm{R}^m$, this approach is used in, e.g.,~\cite{kostova2018model}.
%}
We emphasize, however, that using this approach would increase by $\Delta t^{-1}$ the magnitudes of the eigenvalues on the right-hand side of the POD approximation properties \eqref{eqn:POD_proj_error_W_norm} and \eqref{eqn:POD_proj_error_W_norm_withDQs}, which would yield suboptimal error estimates.  Similar conclusions were reached in Remark 2.3 in~\cite{iliescu2014are} for the case $W=\mathcal{H}$.
\end{remark}

%\blue{[I removed this: ``In fact, when $W=\mathcal{H}$, this approach is used in, e.g.,~\cite{KV02,volkwein2013proper}.'' In (3.2) in \cite{KV02} they use variable weights $ \{ \alpha_j \} $ and in \cite{volkwein2013proper} they use variable weights in Section 4.1 on page 21.]}

%\medskip

%\blue{Do we include a discussion of pointwise in time estimates in past literature?  Should we also say how other noDQ people got away without using this assumption?}

\begin{remark}[Similar Assumptions]
For $W = \mathcal{H}$, Assumption~\ref{ass:LinftyTime} is Assumption 2.1 in~\cite{iliescu2014are} (in which the $L^2$ inner product should be replaced with the correct $\mathcal{H}$ inner product).
A similar assumption (but for the $L^{2}$ projection of a continuous solution on $X^r$ when $\mathcal{H}=L^2$) is made in Assumption 3.2 in~\cite{iliescu2014variational}.
No such assumption is made in~\cite{iliescu2013variational}, since Theorem 3.5 proves an estimate for the average error, not for the pointwise in time error. 
Finally, we note that Figure 4 in~\cite{iliescu2014are} provided numerical validation for Assumption~\ref{ass:LinftyTime} 
for the particular setting 
%considered 
in~\cite{iliescu2014are}
%\blue{for example collections of snapshots} \red{(What do you mean by ``example collections of snapshots?")} 
when $W = \mathcal{H}$.
\end{remark}

%%------------------------------------------------------------------------------------------------
% \subsection{POD Pointwise Error Estimates Without Difference Quotients}
\subsection{Pointwise Error Estimates: noDQ Case}
%Without Difference Quotients}
\label{sec:pointwise-error-noDQ}
%%------------------------------------------------------------------------------------------------

First, we note that in general the scaling factor $ N+1 $ is the worst case scenario for the failure of Assumption~\ref{ass:LinftyTime}.  To see this, note that for any fixed $ k $ we have
\begin{align}
  \| u^k - P_r u^k \|^2_W  &=  (N+1) \frac{1}{N+1} \| u^k - P_r u^k \|^2_W\nonumber\\
    &\leq  (N+1) \left( \frac{1}{N+1} \sum_{i=0}^N \| u^i - P_r u^i \|^2_W \right)\label{eqn:noDQ_pointwise_general_inequality}\\
    &=  (N+1) \sum_{i = r+1}^s \lambda_i^{noDQ} \| \varphi_i \|_W^2\label{eqn:noDQ_pointwise_general_bound},
\end{align}
where we used Lemma \ref{lemma:POD_approx_errors_Wnorm} to obtain 
%the final equality 
\eqref{eqn:noDQ_pointwise_general_bound}.  Note that for many collections of snapshots $ \{ u^k \}_{k=0}^N $ the inequality in \eqref{eqn:noDQ_pointwise_general_inequality} will be very conservative.  Nevertheless, we show below that the above $ N+1 $ scaling is attained for a family of examples.

Assumption~\ref{ass:LinftyTime} says that the error at any particular index is not much larger than the other pointwise errors, or equivalently the inequality \eqref{eqn:noDQ_pointwise_general_inequality} is overly conservative.  Therefore, Assumption~\ref{ass:LinftyTime} will be false if there is an index $ n $ such that the projection error at index $ n $ is much larger than the remaining pointwise errors, i.e.,
\begin{equation}\label{eqn:counterexample_motivation_ineq}
  \| u^n - P_r u^n \|_W^2  \gg  \| u^i - P_r u^i \|_W^2,  \quad  \forall i \neq n,  \quad 0 \leq i \leq N.
\end{equation}
Next, we provide a {\it family of counterexamples} to Assumption~\ref{ass:LinftyTime}, i.e., a family of exact solutions (data) that yield POD bases that satisfy condition~\eqref{eqn:counterexample_motivation_ineq}.

Let $ \{ \varphi_k \}_{k \geq 1} $ be an orthonormal set in a Hilbert space $ \mathcal{H} $, with $ \mathrm{dim}(\mathcal{H}) \geq N+1 $, and let $ \lambda_1 \geq \lambda_2 \geq \cdots > 0 $ be any sequence of positive numbers.  Suppose the data $ U = \{ u^n \}_{n=0}^N \subset \mathcal H $ is given by
\begin{equation}\label{eqn:counterexample_data}
  u^n = (N+1)^{1/2} \lambda_{n+1}^{1/2} \varphi_{n+1},  \quad  n = 0, \ldots, N.
\end{equation}
It can be checked that this data has POD eigenvalues $ \{ \lambda_k \} $ with corresponding POD modes $ \{ \varphi_k \} $.

Let $ W $ be a real Hilbert space with $ U \subset W $. In Proposition~\ref{prop:counterexample}, we show that Assumption~\ref{ass:LinftyTime} fails for the data above.  Specifically, \eqref{eqn:counterexample_exact_error2} shows that the assumption fails for the specific case of $ r = N $ at index $ N $.  Furthermore, if the values $ \{ \lambda_k \} $ decay exponentially fast as in \eqref{eqn:exp_decay_counterexample}, then \eqref{eqn:counterexample_bound_below} shows that the assumption fails for any $ r $ at index $ r $.
\begin{proposition}\label{prop:counterexample}
  Let the data $ U = \{ u^n \}_{n=0}^N \subset \mathcal H $ be given in \eqref{eqn:counterexample_data} as described above.  Then the POD pointwise projection error for $ u^N $ is given by
  \begin{equation}\label{eqn:counterexample_exact_error2}
    \left\| u^{N} - P_N u^{N} \right\|_W^2 = (N+1) \lambda_{N+1} \| \varphi_{N+1} \|_W^2.
  \end{equation}
  Also, for any fixed $ r $ if
  \begin{equation}\label{eqn:exp_decay_counterexample}
      \lambda_k = \beta \| \varphi_k \|_W^{-2} e^{-\gamma k},  \quad  k > r,
  \end{equation}
  for some positive constants $ \beta $ and $ \gamma $, then
  \begin{equation}\label{eqn:counterexample_bound_below}
      \left\| u^r - P_r u^r \right\|_W^2  \geq  \frac{\min\{ 1, \gamma \}}{2} \, (N+1) \sum_{k=r+1}^{N+1} \lambda_k \| \varphi_k \|_W^2.
  \end{equation}
\end{proposition}
\begin{remark}
  Note that for the second part of the result we still assume the POD eigenvalues in \eqref{eqn:exp_decay_counterexample} are ordered so that $ \lambda_1 \geq \lambda_2 \geq \cdots > 0 $.  Depending on the values of $ \| \varphi_k \|_W $ and $ \gamma $, the POD eigenvalues in \eqref{eqn:exp_decay_counterexample} may not be ordered in this way.  In such a case, the POD eigenvalues may need to be reordered in order to obtain a similar result.  If $ W = \mathcal{H} $ or if $ \| \varphi_k \|_W $ increases slowly relative to $ e^{-\gamma k } $, then the ordering $ \lambda_1 \geq \lambda_2 \geq \cdots > 0 $ will automatically be satisfied.
\end{remark}
\begin{proof}
Note that $ P_r u^k = 0 $ when $ k \geq r $ and so
\begin{equation}\label{eqn:counterexample_exact_error1}
  \| u^k - P_r u^k \|_W^2 = (N+1) \lambda_{k+1} \| \varphi_{k+1} \|_W^2,  \quad  k \geq r.
\end{equation}
%
%This observation is key below.
Thus, \eqref{eqn:counterexample_exact_error2} follows immediately from \eqref{eqn:counterexample_exact_error1} with $ k = N $.  
  
  %\medskip
  
  Next, to prove \eqref{eqn:counterexample_bound_below}, fix $ r $ and assume \eqref{eqn:exp_decay_counterexample} holds.  Then \eqref{eqn:counterexample_exact_error1} with $k=r$ gives
  \begin{equation}\label{eqn:counterexample_exact_error3}
    \left\| u^r - P_r u^r \right\|_W^2 = (N+1) \lambda_{r+1} \| \varphi_{r+1} \|_W^2.
  \end{equation}
  We bound half of the right-hand side of~\eqref{eqn:counterexample_exact_error3} from below by a constant multiple of the remaining terms in the sum in \eqref{eqn:counterexample_bound_below}.  Note that the assumption \eqref{eqn:exp_decay_counterexample} on the value of $ \lambda_{r+1} $ gives
  \begin{equation}\label{eqn:counterexample_exact_error4}
    \frac{1}{2} \lambda_{r+1} \| \varphi_{r+1} \|_W^2  =  \frac{\beta}{2} e^{ -\gamma (r+1) }.
  \end{equation}
  Next, we note that the exponential term on the right-hand side of~\eqref{eqn:counterexample_exact_error4} satisfies the following estimate:
  \begin{equation}\label{eqn:counterexample_exact_error4b}
    \frac{1}{\gamma} e^{ -\gamma (r+1) }  \geq  \frac{1}{\gamma} \left( e^{ -\gamma (r+1) } - e^{ -\gamma(N+1)} \right) = \int_{r+1}^{N+1} e^{-\gamma x} dx  \geq  \sum_{k=r+2}^{N+1} e^{-\gamma k}.
  \end{equation}
  Using~\eqref{eqn:exp_decay_counterexample},~\eqref{eqn:counterexample_exact_error4}, and~\eqref{eqn:counterexample_exact_error4b}, we obtain
%  This gives
    \begin{equation}\label{eqn:counterexample_exact_error4c}
    \frac{1}{2} (N+1) \lambda_{r+1} \| \varphi_{r+1} \|_W^2  \geq  \frac{\gamma \beta}{2} (N+1) \sum_{k = r+2}^{N+1} e^{-\gamma k}  =  \frac{\gamma}{2} (N+1) \sum_{k=r+2}^{N+1} \lambda_k \| \varphi_k \|_W^2.
  \end{equation}
  Using~\eqref{eqn:counterexample_exact_error3} and \eqref{eqn:counterexample_exact_error4c}, we get
    \begin{align}\label{eqn:counterexample_exact_error4d}
    \left\| u^r - P_r u^r \right\|_W^2 
    &\geq \frac{1}{2} (N+1) \lambda_{r+1} \| \varphi_{r+1} \|_W^2
    + \frac{\gamma}{2} (N+1) \sum_{k=r+2}^{N+1} \lambda_k \| \varphi_k \|_W^2
    \nonumber \\
    &\geq \frac{\min\{ 1, \gamma \}}{2} \, (N+1) \sum_{k=r+1}^{N+1} \lambda_k \| \varphi_k \|_W^2,
  \end{align}
  which proves \eqref{eqn:counterexample_bound_below}.
%
% \purple{
% There should be = or $\geq$ in the first line of \eqref{eqn:counterexample_exact_error4d}.}
%%  
%%  Next, to prove \eqref{eqn:counterexample_blowup} we show the error blows up for $ k = r $.  Equations \eqref{eqn:counterexample_exact_error3}-\eqref{eqn:counterexample_exact_error4} give
%%  $$
%%    \left\| u^r - P_r u^r \right\|_W^2 = \beta (N+1) e^{ -\gamma (r+1) }.
%%  $$
%%  Since $ r $ is fixed and $ N \to \infty $ as $ \Delta t \to 0 $, we have \left\| u^r - P_r u^r \right\|_W^2 \to \infty $ as $ \Delta t \to 0 $.
%
\end{proof}
Proposition~\ref{prop:counterexample} yields a family of counterexamples to Assumption~\ref{ass:LinftyTime}.
Next, we consider two counterexamples that we investigate numerically in Section~
\ref{sec:numerical-results}.

\subsubsection{Counterexample 1}
    \label{sec:cex-1}

To construct the first counterexample to Assumption~\ref{ass:LinftyTime} (which we denote counterexample 1), we follow the theoretical setting in this section and construct a family of ROM basis functions that satisfy equation~\eqref{eqn:counterexample_data}.
Specifically, we consider an orthonormal set $ \{ \varphi_n \}_{n = 0}^{N} $ in $ \mathcal{H} = L^2(0,1) $ given by
\begin{eqnarray}
	\varphi_{n+1}(x) 
	:= 2^{1/2} \, \sin ((k \, t_{n} + 1) \pi \, x ) \, , 
	\label{eqn:pointwise-error-noDQ-1}
\end{eqnarray}
where $k$ is a positive integer, $x \in [0,1]$, and $t_n = n \, \Delta t$ is chosen such that $k \, t_n \in \mathbbm{N}, \, \forall \, n \in \mathbbm{N}$. 
%\purple{time instance where we evaluate the counterexample 1.} 
%\blue{Finish. Are the POD basis functions orthonormal?} 
%\purple{Yes, they are.}
Next, we choose the eigenvalues
\begin{eqnarray}
	\lambda_1 
	= \lambda_2 
	= \cdots 
	= \lambda_{N+1} 
	=  \displaystyle \frac{1}{2(N+1)} \, , 
	\label{eqn:pointwise-error-noDQ-2}
\end{eqnarray}
which satisfy $ \lambda_1 \geq \lambda_2 \geq \cdots \geq \lambda_{N+1} > 0 $.
Finally, choosing the analytical solution
\begin{eqnarray}
    u_{\text{counterexample 1}}(x,t) 
    = \sin ((k \, t + 1) \pi \, x ) 
    \label{eqn:cex-1} 
\end{eqnarray}
yields the data $ U = \{ u^n \}_{n=0}^N $ that satisfies equation~\eqref{eqn:counterexample_data}.
In Section~\ref{sec:numerical-results}, we investigate numerically counterexample 1 given by the analytical solution~\eqref{eqn:cex-1}. 

\begin{remark}
Equation~\eqref{eqn:counterexample_data} (see also the comment below Assumption A.1 in \cite{mohebujjaman2017energy}) shows that the ROM basis functions are scaled versions of the snapshots.
For counterexample 1, this scaling is illustrated in~\eqref{eqn:pointwise-error-noDQ-1} and \eqref{eqn:cex-1}.
\end{remark}

\subsubsection{Counterexample 2}
    \label{sec:cex-2}

To construct the second counterexample to Assumption~\ref{ass:LinftyTime} (which we denote counterexample 2), we construct a family of ROM basis functions that satisfy both  equation~\eqref{eqn:counterexample_data} and equation~\eqref{eqn:exp_decay_counterexample} in Proposition~\ref{prop:counterexample}.
Specifically, we consider the same orthonormal set $ \{ \varphi_n \}_{n = 0}^{N} $ in $\mathcal{H} = L^2(0,1) $ given in \eqref{eqn:pointwise-error-noDQ-1} above, where again $k$ is a positive integer, $x \in [0,1]$, and $t_n = n \, \Delta t$ is chosen such that $k \, t_n \in \mathbbm{N}, \, \forall \, n \in \mathbbm{N}$. 
Next, for positive constants $ \alpha $, $ \delta $, and $ \rho $, with $ \delta = \rho \Delta t $, we choose exponentially decaying eigenvalues as in \eqref{eqn:exp_decay_counterexample}:
\begin{align*}
    \lambda_{n+1} &= \beta e^{-\gamma(n+1)},\\
	\beta &= \frac{1}{4 \delta (N+1)} e^{-\alpha + \alpha \delta^{-1} \Delta t} = \frac{1}{4 \rho T_1} e^{-\alpha + \alpha \rho^{-1}},\\
	\gamma &= \alpha \delta^{-1} \Delta t = \alpha \rho^{-1},
	\label{eqn:pointwise-error-noDQ-4}
\end{align*}
where $ T_1 = T + \Delta t $.
Finally, it can be checked that choosing the analytical solution 
\begin{eqnarray}
  u_{\text{counterexample 2}}(x,t) 
  = \frac{1}{\sqrt{2 \delta} } \left( e^{-\alpha(1+ t/\delta)} \right)^{1/2} \sin( (kt+1)\pi x )
  \label{eqn:cex-2}
  \end{eqnarray}
yields the data $ U = \{ u^n \}_{n=0}^N $ that satisfies equation~\eqref{eqn:counterexample_data}, which 
% Equations~\eqref{eqn:counterexample_data} and \eqref{eqn:cex-2} 
% \purple{
% Equations~\eqref{eqn:pointwise-error-noDQ-1} and \eqref{eqn:cex-2} 
% }
shows that, in counterexample 2, the ROM basis functions are scaled versions of the snapshots.
In Section~\ref{sec:numerical-results}, we investigate numerically counterexample 2 given by the analytical solution~\eqref{eqn:cex-2}.

%%------------------------------------------------------------------------------------------------
\subsection{POD Pointwise Error Estimates: DQ Case} 
%With Difference Quotients}
\label{sec:pointwise-error-DQ}
%%------------------------------------------------------------------------------------------------
We now give one of the main results of this paper.
%\red{
In Theorem \ref{thm:uniform_estimates}, we show that Assumption~\ref{ass:LinftyTime} is always satisfied in the DQ case.
This will allow us to prove in Section~\ref{sec:ErrEst} optimal pointwise in time ROM error bounds in the DQ case.
In particular, Theorem \ref{thm:uniform_estimates} will show that the assumptions similar to Assumption \ref{ass:LinftyTime} that have been made in, e.g.,~\cite{iliescu2014are}, are unnecessary for obtaining optimal error bounds in the DQ case.
%}
%In past works involving DQs, an assumption similar to Assumption \ref{ass:LinftyTime} has been made for the POD projection error. However, we see in Theorem \ref{thm:uniform_estimates} below that such an assumption is unnecessary for obtaining such error estimates in the DQ case.
%when we use difference quotients.

In continuous time, it is well-known that the magnitude of a function $ z \in H^1(0,T) $ at any point in time is bounded above by a constant multiple of the $ H^1(0,T) $ norm of $ z $.  The constant in the bound only depends on $ T $, and there is also a similar inequality that holds for functions taking values in a Banach space $ Z $ (see, e.g., \cite[Section 5.9.2, page 302, Theorem 2 (iii)]{Evans10}).  Below, we establish a discrete time analogue of this Sobolev embedding $ H^1(0,T;Z) \hookrightarrow C([0,T];Z) $, where the DQs replace the time derivative in the $ H^1(0,T;Z) $ norm.  This lemma will allow us to directly establish POD pointwise projection error bounds in Theorem~\ref{thm:uniform_estimates}, which shows that Assumption~\ref{ass:LinftyTime} is automatically satisfied in the DQ case.
%for POD with difference quotients.
%
\begin{lemma}[Discrete time Sobolev inequality]\label{lemma:discrete_time_Sobolev_embedding}
  Let $ T > 0 $, $ Z $ be a normed space, $ \{ z^n \}_{n=0}^N \subset Z $, and $ \Delta t = T/N $.  Then
  $$
    \max_{0 \leq k \leq N} \| z^k \|_Z^2  \leq  C \left(  \frac{1}{2N+1} \sum_{n=0}^N \left\| z^n \right\|_Z^2 + \frac{1}{2N+1} \sum_{n=1}^{N} \left\| \partial z^n \right\|_Z^2 \right),
  $$
  where $ C = 6 \max\{1,T^2 \} $ and $ \partial z^n = ( z^n - z^{n-1} )/\Delta t $ for $ n = 1, \ldots, N $. 
\end{lemma}
\begin{proof}%[(New)]
  For each $ k, \ell $ with $ N \geq k > \ell \geq 0 $, we have $ z^k - z^\ell = \Delta t \sum_{n=\ell+1}^k \partial z^n $.  This gives
  \begin{equation}\label{eqn:initial pointwise_bound}
    \| z^k \|_Z  \leq  \| z^\ell \|_Z + \sum_{n=1}^N  \Delta t^{1/2}  ( \Delta t^{1/2} \| \partial z^n \|_Z )  \leq  \| z^\ell \|_Z + T^{1/2} \left( \sum_{n=1}^N \Delta t \| \partial z^n \|_Z^2 \right)^{1/2},
  \end{equation}
  where we used $ \sum_{n=1}^N \Delta t = N \Delta t = T $.  This inequality is also clearly true for $ k = \ell $, and a similar argument shows that this inequality also holds for $ 0 \leq k < \ell \leq N $.
  
  Now we choose $ \ell $ so that
  \begin{equation}\label{eqn:minimizing_index}
    \| z^\ell \|_Z = \min_{0 \leq n \leq N} \| z^n \|_Z.
  \end{equation}
  We know such an $ \ell $ must exist since $ N $ is finite.  Then
  \begin{align*}
      \| z^\ell \|_Z &= \frac{1}{N+1}(N+1) \| z^\ell \|_Z = \frac{1}{N+1} \sum_{n=0}^N \| z^\ell \|_Z \\
        &\leq  \frac{1}{T} \sum_{n=0}^N \Delta t \| z^n \|_Z \leq T^{-1/2} \left( \sum_{n=0}^N \Delta t \| z^n \|_Z^2 \right)^{1/2},
  \end{align*}
  where we used \eqref{eqn:minimizing_index}, $ 1/(N+1) < 1/N = T^{-1} \Delta t $, $ \sum_{n=1}^N \Delta t = N \Delta t = T $, and the Cauchy-Schwarz inequality.
  Using this inequality with \eqref{eqn:initial pointwise_bound} yields
  \begin{equation}
    \| z^k \|_Z  \leq  T^{-1/2} \left( \sum_{n=0}^N \Delta t \| z^n \|_Z^2 \right)^{1/2} + T^{1/2} \left( \sum_{n=1}^N \Delta t \| \partial z^n \|_Z^2 \right)^{1/2}.
  \end{equation}
  Squaring both sides, and using the inequalities $ (a+b)^2 \leq 2 (a^2 + b^2) $ and $ \Delta t = (2T+\Delta t)/(2N+1) \leq 3T/(2N+1) $, we obtain the result.
\end{proof}
%

%
% \begin{proof}[(Old)]
% %
%   For each $ k, \ell $ with $ N \geq k > \ell \geq 0 $, we have $ z^k - z^\ell = \Delta t \sum_{n=\ell+1}^k \partial z^n $.  This gives
%   $$
%     \| z^k - z^\ell \|_Z  \leq  \sum_{n=1}^N  \Delta t^{1/2}  ( \Delta t^{1/2} \| \partial z^n \|_Z )  \leq  T^{1/2} \left( \sum_{n=1}^N \Delta t \| \partial z^n \|_Z^2 \right)^{1/2},
%   $$
%   where we used $ \sum_{n=1}^N \Delta t = N \Delta t = T $.  This inequality is also clearly true for $ k = \ell $, and a similar argument shows that this inequality also holds for $ 0 \leq k < \ell \leq N $.  Then use $ \| z^k \|_Z \leq \| z^\ell \|_Z + \| z^k - z^\ell \|_Z $, multiply by $ \Delta t $, sum over $ \ell $, and use the above inequality to obtain
%   %
%   \begin{align*}
%   %
%     \sum_{\ell=1}^N \Delta t \| z^k \|_Z  &\leq  \sum_{\ell=1}^N \Delta t \| z^\ell \|_Z + \sum_{\ell=1}^N \Delta t \| z^k - z^\ell \|_Z  \\
%     %
%       &  \leq  T^{1/2} \left( \sum_{\ell=0}^N \Delta t \| z^\ell \|_Z^2 \right)^{1/2} + T^{3/2} \left( \sum_{n=1}^N \Delta t \| \partial z^n \|_Z^2 \right)^{1/2}.
%   %
%   \end{align*}
%   %
%   Divide by $ \sum_{\ell=1}^N \Delta t = T $, square both sides, and use the inequalities $ (a+b)^2 \leq 2 (a^2 + b^2) $ and $ \Delta t = (2T+\Delta t)/(2N+1) \leq 3T/(2N+1) $ to obtain the result.
% %
% \end{proof}
% %

%
\begin{theorem}
\label{thm:uniform_estimates}
  Let $ X^r = \mathrm{span}\{ \varphi_i \}_{i=1}^r \subset \mathcal{H} $, let $ P_r : \mathcal{H} \to \mathcal{H}$ be the orthogonal projection onto $ X^r $ as defined in \eqref{eqn:POD_projection_P_r}, and let $ s $ be the number of positive POD eigenvalues for $ U^\mathrm{DQ} $.  If $ W $ is a real Hilbert space with $ U^\mathrm{DQ} \subset W $ and $ R_r : W \to W $ is a bounded linear projection onto $ X^r $, then% for $ k = 0, \ldots, N $ we have
  \begin{subequations}\label{eqn:POD_DQs_pointwise_bound}
  \begin{align}
	\max_{0 \leq k \leq N} \left\| u^k - P_r u^k \right\|_\mathcal{H}^2  &\leq  C \sum_{i = r+1}^s \lambda_i^\mathrm{DQ},\label{eqn:POD_DQs_pointwise_bound1}\\
	\max_{0 \leq k \leq N} \left\| u^k - P_r u^k \right\|_W^2  &\leq  C \sum_{i = r+1}^s \lambda_i^\mathrm{DQ} \| \varphi_i \|_W^2,\label{eqn:POD_DQs_pointwise_bound2}\\
    \max_{0 \leq k \leq N} \left\| u^k - R_r u^k \right\|_W^2  &\leq  C \sum_{i = r+1}^s \lambda_i^\mathrm{DQ} \| \varphi_i - R_r \varphi_i \|_W^2,\label{eqn:POD_DQs_pointwise_bound3}
  \end{align}
  \end{subequations}
  where $ C = 6 \max\{1,T^2 \} $.
\end{theorem}
\begin{proof}
  First, note that \eqref{eqn:POD_DQs_pointwise_bound1} follows from \eqref{eqn:POD_DQs_pointwise_bound2} with $ W = \mathcal{H} $ since $ \| \varphi_i \|_\mathcal{H} = 1 $ for all $ i $.  Also, \eqref{eqn:POD_DQs_pointwise_bound2} follows from \eqref{eqn:POD_DQs_pointwise_bound3} since $ P_r \varphi_i = 0 $ for $ i > r $.  Therefore, we only prove \eqref{eqn:POD_DQs_pointwise_bound3}.

  Set $ Z = W $ and $ z^n = u^n - R_r u^n $ for each $ n $.  Using Lemma \ref{lemma:discrete_time_Sobolev_embedding}, $ \partial z^n = \partial u^n - R_r \partial u^n $ for each $ n $, and 
  %\purple{\eqref{eqn:POD_proj_error_W_norm2_withDQs} in} 
  Lemma \ref{lemma:POD_approx_errors_Wnorm_withDQs} gives the result.
\end{proof}

\section{Pointwise Error Estimates: DQ Case%with Difference Quotients
}\label{sec:ErrEst}

In this section, we prove pointwise in time error estimates for the heat equation and discuss the time and ROM discretization optimality of these estimates.  In Section \ref{subsubsec:VDQ}, we prove the pointwise in time error estimates using Crank-Nicolson time stepping in the DQ case
%when a POD basis is constructed with DQs as in 
(see Section \ref{sec:POD_diffQ}). In Section \ref{sec:optimality}, we consider three definitions of optimality for the ROM discretization error and classify the optimality types of each pointwise error estimate in Section \ref{subsubsec:VDQ}. We show that all of the error estimates are optimal in some sense; although, in some cases we need to assume various POD projection uniform boundedness conditions are satisfied.
We also briefly discuss error estimates and optimality for the noDQ case; see Remarks \ref{remark:noDQ_error_estimates_approach},  \ref{remark:noDQ_error_estimates}, and \ref{remark:noDQ_optimality_defs}.  Below, we consider the DQ case unless explicitly mentioned otherwise.

We begin by establishing notation, definitions, and giving preliminary results that will be used in the ensuing analysis. We let $\Omega \in \mathbb{R}^{d}, d = 2,3$ be a regular open domain with Lipschitz continuous boundary $\Omega$ and denote by $(\cdot,\cdot)_{L^{2}}$ and $\|\cdot\|_{L^{2}}$ the $L^{2}$ inner product and norm respectively. We define the function space $X=H_{0}^{1}(\Omega)$ as:
\begin{equation*}
    X:=H_{0}^{1}(\Omega)^{d} = \{v \in H^{1}(\Omega)^{d} : v|_{\Gamma} = 0\}.
\end{equation*}
With the inner product $ (u,v)_{H^1_0} = (\nabla u, \nabla v)_{L^2} $, the space $ X = H^1_0(\Omega) $ is a Hilbert space.

For simplicity, we will only consider the heat equation~\eqref{eqn:heat}. We take $u(\cdot,t)\in X$, $t\in [0,T]$ to be the weak solution of the weak formulation of the heat equation with homogeneous Dirichlet boundary conditions:
\begin{align}\label{eq:heat}
(\partial_{t}u,v)_{L^2}+\nu(\nabla u,\nabla v)_{L^2}=(f,v)_{L^2}\quad \forall v\in X.
\end{align}

Replacing the unknown $u$ with $u_{r}$ in the heat equation \eqref{eq:heat}, using the Galerkin method, projecting the resulting  equations onto a space $X^{r}\subset X$, and discretizing in time using Crank-Nicolson (CN), one obtains the standard CN POD-G-ROM for the heat equation: 
\begin{align}\label{eq:POD-G-ROM}
(\partial u^{n+1}_{r},v_r)_{L^2}+\nu(\nabla u^{n+1/2}_r,\nabla v_r)_{L^2}=(f^{n+1/2},v_r)_{L^2}\quad \forall v_{r}\in X^r,
\end{align}
where $ \partial u^{n+1}_r = ( u^{n+1}_r - u^n_r )/ \Delta t $.  Also, here and below we use the notation $ z^{n+1/2} $ for any discrete or continuous time function $ z $ to denote the average
$$
  z^{n+1/2} := \frac{1}{2} \left( z^{n+1} + z^n \right).
$$
Note that, for continuous time functions, we do \textit{not} use $ z^{n+1/2} $ to denote $ z(t_n+\Delta t/2) $.

\begin{remark}
An alternative CN approach to the time discretization is to replace $ f^{n+1/2} $ in \eqref{eq:POD-G-ROM} with $ f(t_n+\Delta t/2) $. The results in this section also hold for this case.
\end{remark}

% \begin{definition}[Generic Constant $C$]
% For clarity, in what follows, we will denote by $C$ a generic positive constant that may vary from a line to another, but which is always independent of $r$ and $\Delta t$.
% 	\label{def:constant}
% \blue{I think we should move this definition up.}
% \end{definition}

We now prove error estimates for the error $u^{n+1}-u^{n+1}_r$, where $u^{n+1}:=u(t_{n+1})$ is the solution of the weak formulation of the heat equation \eqref{eq:heat}, and $u^{n+1}_r$ is the solution of the CN POD-G-ROM \eqref{eq:POD-G-ROM}. For clarity of presentation, we only consider the error components corresponding to the POD truncation and time discretization, i.e., we ignore the spatial discretiztion (e.g., FE) error.
%\smallskip
We start by noting that the weak solution of the heat equation evaluated at time $ t = t_n+\Delta t/2 $ satisfies:
\begin{align}\label{eq:EE}
\left(\partial u^{n+1},v_r\right)_{L^2}+\nu(\nabla u^{n+1/2},\nabla v_r)_{L^2} = (f^{n+1/2},v_{r})_{L^{2}} + \tau_n(v_{r})\quad \forall v_{r}\in X^r,
\end{align}
where $ \partial u^{n+1} = ( u^{n+1} - u^n )/ \Delta t $ and, after integrating by parts, the consistency error is given by
\begin{align}
\begin{aligned}\label{eqn:tau_CN_consistency error}
    \tau_n(v) &:= \left( \partial u^{n+1} - \partial_t u(t_n+\Delta t/2), v \right)_{L^2} + \nu \left( \Delta ( u(t_n+\Delta t/2) - u^{n+1/2} ),  v \right)_{L^2}\\
      &\quad  +  \left( f(t_n+\Delta t/2) - f^{n+1/2}, v \right)_{L^2}.
\end{aligned}
\end{align}
We assume that the solution $ u $ and the forcing $ f $ are smooth enough so that $ \tau_n(v) $ is well defined for any $ v \in X $.  We provide a more precise regularity assumption below.

The error is split into two parts:
\begin{align}\label{eq:error}
e^{n+1}=u^{n+1}-u^{n+1}_r=(u^{n+1}-w_r^{n+1})-(u^{n+1}_r-w_r^{n+1})=\eta^{n+1}-\phi^{n+1}_r,
\end{align}
where $w_r^{n+1}$ is a proper projection of $u^{n+1}$ on $X^r$, $\eta^{n+1}:=u^{n+1}-w_r^{n+1}$, and $\phi^{n+1}_r=u^{n+1}_r-w_r^{n+1}$. Subtracting \eqref{eq:POD-G-ROM} from \eqref{eq:EE} then yields:
\begin{align}\label{eq:EE1}
\begin{aligned}
(\partial \phi^{n+1}_r ,v_r)_{L^2}+\nu(\nabla \phi^{n+1/2}_r,\nabla v_r)_{L^2}&=(\partial \eta^{n+1},v_r)_{L^2}+\nu(\nabla \eta^{n+1/2},\nabla v_r)_{L^2}
\\
&\quad - \tau_n\left( v_{r} \right) \quad \forall v_{r}\in X^r.
\end{aligned}
\end{align}

% \purple{
% \BEQ\label{eq:EE1}
% \begin{aligned}
% (\partial \phi^{n+1}_r ,v_r)_{L^2}+\nu(\nabla \phi^{n+1/2}_r,\nabla v_r)_{L^2}&=(\partial \eta^{n+1},v_r)_{L^2}+\nu(\nabla \eta^{n+1/2},\nabla v_r)_{L^2}
% \\
% &\quad - \tau_n\left( v_{r} \right) \quad \forall v_{r}\in X^r.
% \end{aligned}
% \EEQ
% }

The standard approach used to prove error estimates in this case is to use the Ritz projection \cite{ArXiv17,iliescu2013variational,iliescu2014variational,KV01,KV02,ESAIMProcS18}. This is also the standard approach in the FE context \cite{GiraultRaviart86,layton2008introduction,thomee2006galerkin,wheeler1973priori}.
%\smallskip
Thus, for the ensuing analysis we choose $w_r:=R_r(u)$ in \eqref{eq:error}, where $R_r(u)$ is the Ritz projection of $u$ on $X^r$:
\begin{align}\label{eq:RitzProj}
(\nabla(u-R_r(u)),\nabla v_r)_{L^2}=0\quad \forall v_{r}\in X^r.
\end{align}
We will then denote $\eta_{Ritz}:=u-R_r(u)$.
Using the Ritz projection, \eqref{eq:EE1} then becomes:
\begin{align}\label{eq:EE2}
(\partial \phi^{n+1}_r ,v_r)_{L^2} + \nu({\nabla \phi^{n+1/2}_r } ,\nabla v_r)_{L^2} = (\partial \eta_{Ritz}^{n+1},v_r)_{L^2} - \tau_n\left( v_{r} \right) \quad \forall v_{r}\in X^r,
\end{align}
 where we have used the fact that $(\nabla \eta^{n+1/2}_{Ritz},\nabla v_r)_{L^2}=0$ by~\eqref{eq:RitzProj}.
 %, since $R_r(u)$ is the Ritz projection of $u$ on $X^r$.

%
\begin{remark}\label{remark:noDQ_error_estimates_approach}
  In the noDQ case (see Section \ref{subsec:POD_no_DQs}), a different approach is typically used to prove error estimates; see, e.g., \cite{chapelle2012galerkin,iliescu2014are,iliescu2014variational,singler2014new}. 
  %\blue{and [other references?]}.  
  Instead of the Ritz projection, in the noDQ case we use the $ L^2 $ projection $ \Pi_r^{L^2} $ and take $ w_r^{n+1} = \Pi_r^{L^2} u^{n+1} $.  The term $ \nu (\nabla \eta^{n+1/2}, \nabla v_r )_{L^2} $ in \eqref{eq:EE1} no longer vanishes; instead, the DQ projection error term is eliminated, i.e., $ ( \partial \eta^{n+1}, v_r )_{L^2} = 0 $ in \eqref{eq:EE1}.
  However, as explained in Remark~\ref{remark:noDQ_error_estimates}, the resulting pointwise error estimates are suboptimal.
\end{remark}

For the POD basis construction, we must specify a Hilbert space $ \mathcal{H} $.  For this problem, two natural Hilbert spaces that are often used are $ \mathcal{H} = L^2(\Omega) $ or $ \mathcal{H} = X = H^1_0(\Omega) $.  Let $ X^r $ be the span of the first $ r $ POD modes for the data set containing the snapshots $ \{ u^n \}_{n=0}^N $ and the snapshot DQs $ \{ \partial u^n \}_{n=1}^N $.  We can use Lemma \ref{lemma:POD_approx_errors_Wnorm_withDQs} and Theorem \ref{thm:uniform_estimates} to obtain POD approximation error results with either $ W = L^2(\Omega) $ or $ W = H^1_0(\Omega) $.  We note that in the case $ \mathcal{H} = H^1_0(\Omega) $, the standard orthogonal POD projection $ P_r $ is exactly equal to the Ritz projection $ R_r $.

\subsection{Error estimates}\label{subsubsec:VDQ}
We give multiple error bounds for the solution when both the $L^{2}$ and $H^{1}_0$ POD bases are used.  Specifically, we first provide a pointwise in time error bound for the $ L^2 $ norm of the solution, and an error bound for the solution norm (a discrete time analogue of the $ L^2(0,T;H^1_0(\Omega)) $ norm) that includes the $L^2 $ norm of the solution at the final time step.  Then, we prove a pointwise in time error bound for the $ H^1_0 $ norm of the solution.

We assume the solution $ u $ of the heat equation~\eqref{eqn:heat} and the forcing $ f $ satisfy the regularity condition
\begin{align}\label{eqn:regularity_condition}
  u_{ttt}, \: \Delta u_{tt}, \: f_{tt} \in L^{2}(0,T;L^{2}(\Omega)).
\end{align}

We also define the regularity constants
\begin{align}\label{eqn:regularity_constants}
\begin{aligned}
  I_n(u,f) &:= \| u_{ttt} \|^2_{L^2(t_n,t_{n+1};L^2)} + \| \Delta u_{tt} \|^2_{L^2(t_n,t_{n+1};L^2)} + \| f_{tt} \|^2_{L^2(t_n,t_{n+1};L^2)},\\
  I(u,f) &:= \| u_{ttt} \|^2_{L^2(0,T;L^2)} + \| \Delta u_{tt} \|^2_{L^2(0,T;L^2)} + \| f_{tt} \|^2_{L^2(0,T;L^2)}.
\end{aligned}
\end{align}

\begin{lemma}\label{lemma:pointwise_L2_ROM_bound}
Consider the CN POD-G-ROM scheme \eqref{eq:POD-G-ROM}. If \eqref{eqn:regularity_condition} is satisfied, then the following error bounds hold when the $L^{2}$ POD basis is used
\begin{equation}\label{eq:L2_L2_bound}
\max_{1 \leq k \leq N} \|e^{k}\|_{L^{2}}^{2} \leq C \left(\sum_{i=r+1}^{s} \lambda_{i}^{DQ}\|\varphi_{i} -R_{r}(\varphi_{i})\|^{2}_{L^{2}} + \|\phi^{0}_r\|_{L^2}^{2} + \Delta t^{4} I(u,f) \right),
\end{equation}
% \begin{align}\label{eq:energy_L2_bound}
% \begin{aligned}
% \|e^{N}\|_{L^{2}}^{2} +  \Delta t\sum_{n=0}^{N-1}\|\nabla e^{n+1/2}\|_{L^{2}}^{2} \leq C &\biggr(\sum_{i=r+1}^{s} \lambda_{i}^{DQ}\left(\|\varphi_{i} -R_{r}(\varphi_{i})\|^{2}_{L^{2}} + \|\nabla(\varphi_{i} -R_{r}(\varphi_{i}))\|^{2}_{L^{2}}\right)
% \\
% &+ \|\phi^{0}_r\|_{L^2}^{2} + \Delta t^{4} I(u,f) \biggr),
% \end{aligned}
% \end{align}

%\purple{
\begin{align}\label{eq:energy_L2_bound}
\begin{aligned}
\|e^{N}\|_{L^{2}}^{2} +  \Delta t\sum_{n=0}^{N-1}\|\nabla e^{n+1/2}\|_{L^{2}}^{2} \leq C &\biggr(\sum_{i=r+1}^{s} \lambda_{i}^{DQ} (\|\varphi_{i} -R_{r}(\varphi_{i})\|^{2}_{L^{2}}  \\
& + \|\nabla(\varphi_{i} -R_{r}(\varphi_{i}))\|^{2}_{L^{2}} )
+ \|\phi^{0}_r\|_{L^2}^{2} + \Delta t^{4} I(u,f) \biggr),
\end{aligned}
\end{align}
%}
and the following error bounds hold when the $H^{1}_0$ POD basis is used
\begin{equation}\label{eq:L2_H1_bound}
\max_{1 \leq k \leq N} \|e^{k}\|_{L^{2}}^{2} \leq C \left(\sum_{i=r+1}^{s} \lambda_{i}^{DQ}\|\varphi_{i}\|^{2}_{L^{2}} + \|\phi^{0}_r\|_{L^2}^{2} + \Delta t^{4} I(u,f) \right),
\end{equation}
\begin{equation}\label{eq:energy_H1_bound}
\|e^{N}\|_{L^{2}}^{2} +  \Delta t \sum_{n=0}^{N-1}\|\nabla e^{n+1/2}\|_{L^{2}}^{2} \leq C \left(\sum_{i=r+1}^{s} (1+\|\varphi_{i}\|^{2}_{L^{2}})\lambda_{i}^{DQ} + \|\phi^{0}_r\|_{L^2}^{2} + \Delta t^{4} I(u,f) \right).
\end{equation}
\label{lem:l2-error-estimate}
\end{lemma}
\begin{proof}
 We let $v_r:=\phi^{n+1/2}_r$ in equation \eqref{eq:EE2}, apply Cauchy-Schwarz, Young's, and Poincar{\'e}'s inequalities, and Taylor's theorem\footnote{see, e.g., \cite[Lemma 26, page 166]{layton2008introduction} or \cite[pages 16-17]{thomee2006galerkin}} to yield:
\begin{align}\label{eq:EE3}
\begin{aligned}
\|\phi^{n+1}_r\|_{L^2}^{2} - \|\phi^{n}_r\|_{L^2}^{2} + 2 \nu \Delta t \, \| \nabla \phi_r^{n+1/2} \|_{L^2}^{2} \leq \biggr( & C \Delta t \left\| \partial \eta_{Ritz}^{n+1} \right\|_{L^{2}}^{2} + C \Delta t^{4} I_n(u,f)\\
  &\quad  + \nu \Delta t \, \| \nabla \phi_r^{n+1/2} \|_{ L^2 }^2 \biggr).
\end{aligned}
\end{align}
Now, summing from $n=0$ to $k-1$ gives
\BEQ\label{eq:EE4}
\begin{aligned}
\|\phi^{k}_r\|_{L^2}^{2} + \nu \sum_{n=0}^{k-1} \Delta t \, \| \nabla \phi_r^{n+1/2} \|_{L^2}^2 \leq  C \biggr(&\sum_{n=0}^{k-1} \Delta t \left\| \partial \eta_{Ritz}^{n+1} \right\|_{L^{2}}^{2} + \Delta t^{4} I(u,f) +\|\phi^{0}_r\|_{L^2}^{2}\biggr).
\end{aligned}
\EEQ
By the triangle inequality we have $ \|e^{k}\|^{2}_{L^2} \leq 2(\|\eta_{Ritz}^{k}\|^{2}_{L^2} + \|\phi_{r}^{k}\|^{2}_{L^2})$. Applying this inequality, rearranging terms, dropping an unnecessary term, and taking a maximum among constants it then follows from \eqref{eq:EE4} that
\BEQ\label{eq:EE5}
\begin{aligned}
\|e^{k}\|_{L^{2}}^{2} &\leq C\biggr( \Delta t \sum_{n=1}^{N} \|\partial \eta_{Ritz}^{n}\|_{L^{2}}^{2} + \|\eta_{Ritz}^{k}\|_{L^2}^{2} + \|\phi^{0}_r\|_{L^2}^{2} + \Delta t^{4} I(u,f) \biggr).
\end{aligned}
\EEQ
The pointwise in time estimates \eqref{eq:L2_L2_bound} and \eqref{eq:L2_H1_bound} then follow from applying Lemma \ref{lemma:POD_approx_errors_Wnorm_withDQs}  and Theorem \ref{thm:uniform_estimates} and using $\Delta t (2N+1) = (2+1/N) T \leq 3T$.

The error bounds \eqref{eq:energy_L2_bound} and \eqref{eq:energy_H1_bound} in the solution norm follow by taking $ k = N $ in \eqref{eq:EE4} and proceeding similarly.
\end{proof}

\begin{remark}\label{remark:noDQ_error_estimates}
  We briefly provide one pointwise in time error estimate for the noDQ case with the $ L^2 $ POD basis; other pointwise estimates can be obtained using similar ideas.  In the noDQ case, to obtain a pointwise in time $ L^2 $ error estimate one can proceed in a similar fashion to the above proof using the $ L^2 $ projection instead of the Ritz projection, as discussed in Remark \ref{remark:noDQ_error_estimates_approach}. The error estimate~\eqref{eq:noDQ_error_estimates_L2basis} can be obtained using Lemma \ref{lemma:POD_approx_errors_Wnorm} with $ \mathcal{H} = L^2(\Omega) $ and $ W = H^1_0(\Omega) $, and the worst case pointwise projection error bound \eqref{eqn:noDQ_pointwise_general_bound}:
\begin{equation}
\label{eq:noDQ_error_estimates_L2basis}
\begin{aligned}
\max_{1 \leq k \leq N} \|e^{k}\|_{L^{2}}^{2}
\leq C &\biggr( (N+1) \sum_{i=r+1}^{N+1} \lambda_{i}^{noDQ} + \sum_{i=r+1}^{N+1} \lambda_i^{noDQ} \| \nabla \varphi_i \|^2_{L^2}\\
&\quad + \|\phi^{0}_r\|_{L^2}^{2} + \Delta t^{4} I(u,f) \biggr).
\end{aligned}
\end{equation}
  If Assumption~\ref{ass:LinftyTime} is satisfied, then the $ (N+1) $ scaling factor can be removed.

We emphasize that the error estimate~\eqref{eq:noDQ_error_estimates_L2basis} is suboptimal; see Remark \ref{remark:noDQ_optimality_defs} below for precise optimality definitions. First, the estimate is suboptimal with respect to the time discretization error because of the extra factor $(N+1) = (T\Delta t^{-1} + 1)$.
Second, the estimate is suboptimal with respect to the ROM projection error because of the second term on the right-hand side, which contains $\| \nabla \varphi_i \|^2_{L^2}$ instead of $\| \varphi_i \|^2_{L^2}$.
This is a consequence of using the $L^2$ projection instead of the classical Ritz projection (see Remark \ref{remark:noDQ_error_estimates_approach}).
As explained in~\cite{iliescu2014are}, using the $L^2$ projection eliminates the need to use the DQs, but yields suboptimal estimates with respect to the ROM projection error. 
Thus, even if Assumption~\ref{ass:LinftyTime} is satisfied and the $ (N+1) $ scaling factor can be removed, the error estimate~\eqref{eq:noDQ_error_estimates_L2basis} is still suboptimal.  If the $H^1_0 $ POD basis is used instead, the resulting error estimate is also suboptimal, even if Assumption~\ref{ass:LinftyTime} is satisfied; the details are similar.

\end{remark}

Next, we prove a pointwise in time error bound in the $H^{1}_0$ norm.

\begin{lemma}\label{lemma:pointwise_H1_ROM_bound}
Consider the CN POD-G-ROM scheme \eqref{eq:POD-G-ROM}. If \eqref{eqn:regularity_condition} is satisfied, then the following error bound holds when the $L^{2}$ POD basis is used
\begin{equation}\label{eq:H1_L2_bound}
\begin{aligned}
\max_{1 \leq k \leq N} \|\nabla e^{k}\|_{L^{2}}^{2} \leq C &\biggr(\sum_{i=r+1}^{s} \lambda_{i}^{DQ}\left(\|\varphi_{i} -R_{r}(\varphi_{i})\|^{2}_{L^{2}} + \|\nabla(\varphi_{i} -R_{r}(\varphi_{i}))\|^{2}_{L^{2}}\right) \\
&\quad  + \|\nabla \phi^{0}_r\|_{L^2}^{2} + \Delta t^{4} I(u,f) \biggr) \, ,
\end{aligned}
\end{equation}
and the following error bound holds when the $H^{1}_0$ POD basis is used
\begin{equation}\label{eq:H1_H1_bound}
\max_{1 \leq k \leq N} \|\nabla e^{k}\|_{L^{2}}^{2} \leq C \left(\sum_{i=r+1}^{s} \lambda_{i}^{DQ}(1 + \|\varphi_{i}\|^{2}_{L^{2}}) + \|\nabla \phi^{0}_r\|_{L^2}^{2} + \Delta t^{4} I(u,f) \right).
\end{equation}
\end{lemma}
\begin{proof}
We let $v_{r}:=\partial\phi^{n+1}_r$ in \eqref{eq:EE2}: 
\begin{align}\label{eq:ApprCDQs}
\|\partial\phi^{n+1}_r\|^{2}_{L^{2}} + \frac{\nu}{2 \Delta t}( \| \nabla \phi_r^{n+1} \|_{L^2}^2 - \| \nabla \phi_r^n \|_{L^2}^2 ) = ({\partial\eta_{Ritz}^{n+1}},\partial \phi_r^{n+1})_{L^2} - \tau_{n}( \partial \phi_r^{n+1}).
\end{align}

Applying Cauchy-Schwarz and Young's inequalities along with Taylor's theorem on the RHS of \eqref{eq:ApprCDQs},  we get:
\begin{align}\label{eq:1ApprCDQs}
\begin{aligned}
\nu (\|\nabla \phi^{n+1}_{r}\|^{2}_{L^{2}} - \|\nabla \phi^{n}_{r}\|^{2}_{L^{2}} ) + 2 \Delta t \, \| \partial \phi_r^{n+1} \|_{L^2}^2 \leq &\Delta t \, \| \partial \eta_{Ritz}^{n+1} \|_{L^2}^2 + \frac{3}{2} \Delta t \, \| \partial \phi_r^{n+1} \|_{L^2}^2\\
  &\quad + C \Delta t^4 I_n(u,f).
\end{aligned}
\end{align}

Next, sum from $ n=0 $ to $ n=k-1 $ and drop an unnecessary term:
$$
  \| \nabla \phi_r^k \|_{L^2}^2  \leq  \frac{1}{\nu} \sum_{n=0}^{N-1} \Delta t \, \| \partial \eta_{Ritz}^{n+1} \|_{L^2}^2 + C \Delta t^4 I(u,f) + \| \nabla \phi_r^0 \|_{L^2}^2.
$$

% \purple{
% $$
%   \| \nabla \phi_r^k \|_{L^2}^2  \leq  \frac{1}{\nu} \sum_{n=0}^{k-1} \Delta t \, \| \partial \eta_{Ritz}^{n+1} \|_{L^2}^2 + C \Delta t^4 I(u,f) + \| \nabla \phi_r^0 \|_{L^2}^2.
% $$
% }
Now use $ \| \nabla e^{k}\|^{2}_{L^2} \leq 2 (\| \nabla \eta_{Ritz}^{k}\|^{2}_{L^2} + \| \nabla \phi_{r}^{k}\|^{2}_{L^2})$ to obtain
$$
  \| \nabla e^k \|_{L^2}^2  \leq  C \left( \sum_{n=0}^{N-1} \Delta t \, \| \partial \eta_{Ritz}^{n+1} \|_{L^2}^2 + \| \nabla \eta_{Ritz}^{k}\|_{L^2}^{2} + \Delta t^4 I(u,f) + \| \nabla \phi_r^0 \|_{L^2}^2 \right).
$$
We use Lemma \ref{lemma:POD_approx_errors_Wnorm_withDQs}, Theorem \ref{thm:uniform_estimates}, and $\Delta t (2N+1) = (2+1/N) T \leq 3T$ to complete the proof.
\end{proof}

%%%%%%%%%%%%%%%%%%%%%%%%%%%%%%%%%%%%%
\subsection{Optimality of Pointwise ROM Discretization Errors}
\label{sec:optimality}

Next, we discuss three different definitions of optimality for pointwise in time ROM discretization errors.  Again, we assume we are in the DQ case throughout; although we do briefly discuss the noDQ case in Remark \ref{remark:noDQ_optimality_defs} below.  We classify the optimality type of each pointwise in time error bound for the DQ case from Section \ref{subsubsec:VDQ}. 

The optimality type of a pointwise error bound depends on both the space $ \mathcal{H} $ for the POD basis and the space $ W $ for the pointwise error norm.  In Section \ref{subsubsec:VDQ} we considered four possibilities: we used $ \mathcal{H} = L^2 $ or $ \mathcal{H} = H^1_0 $ for the POD basis, and we used $ W = L^2 $ or $ W = H^1_0 $ for the error norm.  Below, we let $ \mathcal{H} $ and $ W $ be any real Hilbert spaces, we consider the DQ case, and we let $ e^k = u^k - u^k_r $ be the ROM error for $ k = 0, \ldots, N $.  For the discretization, we assume that, if certain conditions are satisfied, then there exists a constant $ C $ so that the following pointwise error bound holds:
\begin{equation}
  \max_{1 \leq k \leq N} \| e^k \|^2_W  \leq  C \left( \Lambda_r + \Lambda^0_r + \zeta(\Delta t) + \xi(h) \right),
\end{equation}
where
\begin{itemize}
    \item $ \Lambda_r $ is the ROM discretization error, and depends only on $ r $, the POD eigenvalues, and the POD modes;
    \item $ \Lambda^0_r $ is the ROM discretization error for the initial condition only, and depends only on $ r $, the POD eigenvalues, and the POD modes;
    \item $ \zeta(\Delta t) $ is an \textit{optimal} time discretization error; and
    \item $ \xi(h) $ is an \textit{optimal} spatial discretization error.
\end{itemize}
We automatically consider the discretization error suboptimal if either the time or space discretization errors are suboptimal; therefore, we assume those errors are optimal here and focus on the ROM discretization error.

  Let $ X^r \subset \mathcal{H} $ be the span of the first $ r $ POD modes, and assume $ X^r $ is also contained in $ W $.  Let $ P_r : \mathcal{H} \to \mathcal{H} $ be the orthogonal POD projection onto $ X^r $, and let $ \Pi^W_r : W \to W $ be the $ W $-orthogonal projection onto $X^r $.  Also, let $ s $ be the number of positive POD eigenvalues.

\begin{definition}
\label{def:optimal}
  We say the ROM discretization error $ \Lambda_r $ is
  \begin{itemize}
      \item \textbf{truly optimal} if there exists a constant $ C $ such that
      \begin{equation}\label{eqn:ROMerror_truly_optimal}
          \Lambda_r  \leq  C \Lambda_r^\star,  \quad  \Lambda_r^\star := \max_{1 \leq k \leq N} \| u^k - \Pi^W_r u^k \|_W^2,
      \end{equation}
      \item \textbf{optimal-I} if there exists a constant $ C $ such that
      \begin{equation}\label{eqn:ROMerror_optimalI}
          \Lambda_r  \leq  C \Lambda_r^I,  \quad  \Lambda_r^I := \sum_{i = r+1}^s \lambda_i \| \varphi_i \|_W^2,
      \end{equation}
      \item \textbf{optimal-II} if there exists a constant $ C $ such that
      \begin{equation}\label{eqn:ROMerror_optimalII}
          \Lambda_r  \leq  C \Lambda_r^{II},  \quad  \Lambda_r^{II} := \sum_{i = r+1}^s \lambda_i \| \varphi_i - \Pi^W_r \varphi_i \|_W^2.
      \end{equation}
  \end{itemize}
  The constant $ C $ above should be independent of all discretization parameters, but may depend on the solution data and the problem data.
\end{definition}

We note that the first two notions of optimality above are generalizations of definitions discussed in \cite{iliescu2014are}, while we believe the optimal-II definition is new.  We discuss each type of optimality below.

\begin{remark}
Note that we do not consider the ROM discretization error for the initial condition, $ \Lambda^0_r $, in these optimality definitions.  These definitions can be modified to include the ROM initial condition error, if desired.
\end{remark}

\textbf{Truly optimal:}  Since $ \Pi^W_r $ is the $W $-orthogonal projection, the quantity $ \Lambda_r^\star $ defined in \eqref{eqn:ROMerror_truly_optimal} is the best possible pointwise POD data approximation error.  As discussed in \cite{iliescu2014are}, this is the most natural definition of optimality; however, it may not be straightforward to evaluate the quantity $ \Lambda_r^\star $ and compare it to the ROM discretization error bound $ \Lambda_r $.

\textbf{Optimal-I} (Optimal type I): Since it may not be easy to deal with the notion of truly optimal, Iliescu and Wang proposed the notion of Optimal-I in \cite{iliescu2014are}.  Optimal-I has the advantage of being simple to compute since $ \Lambda_r^I $ involves only the POD eigenvalues and modes.  Optimal-I is also simple to interpret since from Lemma \ref{lemma:POD_approx_errors_Wnorm_withDQs} we have
\begin{equation}
  \Lambda_r^I  =  \frac{1}{2N+1} \sum_{n=0}^N \left\| u^n - P_r u^n \right\|_W^2 + \frac{1}{2N+1} \sum_{n=1}^{N} \left\| \partial u^n - P_r \partial u^n \right\|_W^2.
  \label{eqn:optimality-1}
\end{equation}
Therefore, $ \Lambda_r^I $ is the \textit{total} POD projection error for all of the data using the POD projection $ P_r $ and the error norm $ W $.

\textbf{Optimal-II} (Optimal type II): The value of $ \Lambda_r^{II} $ is also relatively straightforward to compute, since it involves only POD eigenvalues, modes, and the projection $ \Pi_r^W $. Also, by Lemma \ref{lemma:POD_approx_errors_Wnorm_withDQs} we have
\begin{equation}
  \Lambda_r^{II}  =  \frac{1}{2N+1} \sum_{n=0}^N \left\| u^n - \Pi_r^W u^n \right\|_W^2 + \frac{1}{2N+1} \sum_{n=1}^{N} \left\| \partial u^n - \Pi_r^W \partial u^n \right\|_W^2.
  \label{eqn:optimality-2}
\end{equation}
Since $ \Pi^W_r $ is the $W $-orthogonal projection, the quantity $ \Lambda_r^{II} $ is the best possible \textit{total} POD data approximation error, and~\eqref{eqn:optimality-1}--\eqref{eqn:optimality-2} imply
$$
  \Lambda_r^{II}  \leq  \Lambda_r^I.
$$
Optimal-II has the advantage of using a best possible POD approximation error, while also being relatively simple to compute and understand.  Finally, we note that if $ W = \mathcal{H} $ then $P_r = \Pi_r^W $ and therefore Optimal-I and Optimal-II are identical; however, Optimal-I and Optimal-II may be different if $ \mathcal{H} \neq W $.

\textbf{Comparing the optimality types:}  Since we are in the DQ case, the pointwise POD projection error result Theorem \ref{thm:uniform_estimates} implies that there exists a constant $ C $ such that
$$
  \Lambda_r^\star \leq C \Lambda_r^{II}.
$$
The above definitions, observations, and inequalities give the following result comparing the optimality types.
\begin{proposition}\label{prop:compare_optimality_types}
  The following hold:
  \begin{enumerate}[label=(\roman*)]
      \item If the ROM discretization error is truly optimal, then it is Optimal-II.
      \item If the ROM discretization error is Optimal-II, then it is Optimal-I.
      \item If $ \mathcal{H} = W $, then Optimal-I and Optimal-II are identical conditions.
      \item If there exists a constant $ C $
      %, independent of all discretization parameters, 
      such that
      \begin{equation}\label{eqn:H1_POD_basis_proj_bound}
        \| \varphi_i \|_{W}  \leq  C \| \varphi_i - \Pi^{W}_r \varphi_i  \|_{W},  \quad  r+1 \leq i \leq s,
      \end{equation}
  and if the ROM discretization error is Optimal-I, then it is Optimal-II.
  \end{enumerate}
\end{proposition}
In general, we do not know if Optimal-II implies truly optimal; however, again, $ \Lambda_r^{II} $ is easier to deal with compared to $ \Lambda_r^{\star} $.  We also do not know in general if Optimal-I implies Optimal-II when $ \mathcal{H} \neq W $.  We discuss condition \eqref{eqn:H1_POD_basis_proj_bound} below.

\begin{remark}[The noDQ case]\label{remark:noDQ_optimality_defs}
  In the noDQ case, the same definitions of optimality can be used and Lemma \ref{lemma:POD_approx_errors_Wnorm} also gives interpretations of $ \Lambda_r^I $ and $ \Lambda_r^{II} $ as total POD projections errors in the $ W $ norm.  As in the DQ case, Optimal-II implies Optimal-I, the two conditions are equivalent if $ \mathcal{H} = W $, and Optimal-I with \eqref{eqn:H1_POD_basis_proj_bound} implies Optimal-II.  
  
  \medskip
  
  However, as shown in Proposition~\ref{prop:counterexample}, in general we cannot bound the pointwise POD projection error by a constant multiple of the total POD projection error, i.e., Assumption \ref{ass:LinftyTime} is not always satisfied.
  Thus, we do not know if truly optimal implies Optimal-II.  
  Furthermore, even if Assumption \ref{ass:LinftyTime} is satisfied, the $ L^2 $ pointwise error estimate~\eqref{eq:noDQ_error_estimates_L2basis} in Remark \ref{remark:noDQ_error_estimates} is not optimal in any sense, since the second term on its right-hand side contains $\| \nabla \varphi_i \|^2_{L^2}$ instead of $\| \varphi_i \|^2_{L^2}$.
\end{remark}

\textbf{Optimality of Bounds in Section \ref{subsubsec:VDQ}:}  Next, we consider the optimality type of each pointwise in time error bound for the DQ case from Section \ref{subsubsec:VDQ}.  Comparing the pointwise bounds in Lemmas \ref{lemma:pointwise_L2_ROM_bound} and \ref{lemma:pointwise_H1_ROM_bound} to the above optimality definitions gives the following result.
\begin{theorem}
\label{theorem:opt_equiv}
  For the pointwise error bounds in Lemma \ref{lemma:pointwise_L2_ROM_bound} with error norm $ W = L^2 $:
  \begin{enumerate}[label=(\roman*)]
      \item If the $ L^2 $ POD basis is used (i.e., $ \mathcal{H} = L^2 $) and there exists a constant $ C $
      %, independent of all discretization parameters, 
      such that
      \begin{equation}\label{eqn:Ritz_proj_unif_bound}
        \| \varphi_i - R_r( \varphi_i ) \|_{L^2}  \leq  C,  \quad  r+1 \leq i \leq s,
      \end{equation}
      then the ROM discretization error in \eqref{eq:L2_L2_bound} is Optimal-I (which is identical to Optimal-II).
      \item If the $ H^1_0 $ POD basis is used (i.e., $ \mathcal{H} = H^1_0 $), then the ROM discretization error in \eqref{eq:L2_H1_bound} is Optimal-I.
      \item If the $ H^1_0 $ POD basis is used (i.e., $ \mathcal{H} = H^1_0 $) and condition  \eqref{eqn:H1_POD_basis_proj_bound} is satisfied (with $ W = L^2 $), then the ROM discretization error in \eqref{eq:L2_H1_bound} is Optimal-II.
  \end{enumerate}
  For the pointwise error bounds in Lemma \ref{lemma:pointwise_H1_ROM_bound} with error norm $ W = H^1_0 $:
  \begin{enumerate}[label=(\roman*), start=4]
      \item If the $ L^2 $ POD basis is used (i.e., $ \mathcal{H} = L^2 $), then the ROM discretization error in \eqref{eq:H1_L2_bound} is Optimal-II.
      \item If the $ H^1_0 $ POD basis is used (i.e., $ \mathcal{H} = H^1_0 $), then the ROM discretization error in \eqref{eq:H1_H1_bound} is Optimal-I (which is identical to Optimal-II).
  \end{enumerate}
\end{theorem}
\begin{proof}
Beginning with (i), the ROM discretization error from \eqref{eq:L2_L2_bound} is given by
\begin{equation}
    \Lambda_r =\sum_{i=r+1}^{s} \lambda_{i}^{DQ}\|\varphi_{i} -R_{r}(\varphi_{i})\|^{2}_{L^{2}}.
\end{equation}
By \eqref{eqn:Ritz_proj_unif_bound}, the $L^{2}$ orthonormality of the POD basis, and the definition of Optimal-I it follows that
\begin{equation}
    \Lambda_r \leq C \sum_{i=r+1}^{s}\lambda_{i}^{DQ} = C \sum_{i=r+1}^{s}\lambda_{i}^{DQ} \|\varphi_i\|^{2}_{L^{2}} = C\Lambda_{r}^{I}.
\end{equation}
From Proposition \ref{prop:compare_optimality_types} since $\mathcal{H} = W$ this is identical to Optimal-II.

For (ii) 
the ROM discretization error from \eqref{eq:L2_H1_bound} is given by
\begin{equation}
    \Lambda_r =\sum_{i=r+1}^{s} \lambda_{i}^{DQ}\|\varphi_{i}\|^{2}_{L^{2}},
\end{equation}
which is Optimal-I by definition.

Next, (iii) follows from (ii) and Proposition \ref{prop:compare_optimality_types}.  

For (iv), the ROM discretization error in \eqref{eq:H1_L2_bound} is given by
\begin{equation}
   \Lambda_r = \sum_{i=r+1}^{s} \lambda_{i}^{DQ}\left(\|\varphi_{i} -R_{r}(\varphi_{i})\|^{2}_{L^{2}} + \|\nabla(\varphi_{i} -R_{r}(\varphi_{i}))\|^{2}_{L^{2}}\right).
\end{equation}
%That $\Lambda_r$ is Optimal-II follows from applying the Poincar{\'e} inequality to $\|\varphi_{i} -R_{r}(\varphi_{i})\|^{2}_{L^{2}}$.
Applying Poincar{\'e}'s inequality to $\|\varphi_{i} -R_{r}(\varphi_{i})\|^{2}_{L^{2}}$ shows that $\Lambda_r$ is Optimal-II.

Finally, to prove (v) we use the fact that $ P_r = R_r $ for $ \mathcal{H} = H^1_0 $,  Poincar{\'e}'s inequality, and the fact that $ P_r \varphi_i = 0 $ for $ i > r $ to obtain
\begin{equation*}
   \Lambda_r = C\sum_{i=r+1}^{s} \lambda_{i}^{DQ}\left(\|\varphi_{i} -P_{r}(\varphi_{i})\|^{2}_{L^{2}} + \|\nabla(\varphi_{i} -P_{r}(\varphi_{i}))\|^{2}_{L^{2}}\right) \leq C \sum_{i=r+1}^{s} \lambda_{i}^{DQ}\|\nabla\varphi_{i}\|^{2}_{L^{2}} ,
\end{equation*}
which is Optimal-I by definition. Since $W = \mathcal{H} = H^1_0 $, this is identical to Optimal-II by Proposition \ref{prop:compare_optimality_types}. 
\end{proof}

%The proofs of these statements follow directly from the optimality definitions, the fact that $ P_r = R_r $ for $ \mathcal{H} = H^1_0 $, the Poincar{\'e} inequality, the fact that $ P_r \varphi_i = 0 $ for $ i > r $, and Proposition \ref{prop:compare_optimality_types}.
%\blue{I think we should prove each bullet separately; and maybe enumerate the items.}

The $ W = L^2 $ and $\mathcal{H} = H^1_0 $ case suggests it may be possible for the ROM discretization error to be Optimal-I but not Optimal-II, since an additional assumption is required for Optimal-II.  However, no other case shows a substantial difference between Optimal-I and Optimal-II.  It is possible that further differences arise for other partial differential equations; we leave this to be investigated elsewhere.

We note that equations \eqref{eqn:H1_POD_basis_proj_bound} and  \eqref{eqn:Ritz_proj_unif_bound} are uniform boundedness type conditions for non-orthogonal POD projections.  
%We note that 
Indeed, for the case $ W = \mathcal{H} = L^2 $, the Ritz projection $ R_r : L^2 \to L^2 $ is not orthogonal (even though it is orthogonal when viewed as a mapping $ R_r : H^1_0 \to H^1_0 $).
Thus, \eqref{eqn:Ritz_proj_unif_bound} is a uniform boundedness condition for a non-orthogonal POD projection.
%\blue{So what?  Equation (4.30) is relevant?}
Furthermore, for the case $ W = L^2 $ and $ \mathcal{H} = H^1_0 $, we have $ R_r \varphi_i = 0 $ for $ i > r $, and so \eqref{eqn:H1_POD_basis_proj_bound} can be viewed as
      \begin{equation}\label{eqn:H1_POD_basis_proj_altbound}
        \| \varphi_i - R_r \varphi_i \|_{L^2}  \leq  C \| \varphi_i - \Pi^{L^2}_r \varphi_i  \|_{L^2},  \quad  r+1 \leq i \leq s.
      \end{equation}
      %
%In this case, we need 
Thus, \eqref{eqn:H1_POD_basis_proj_bound} is a uniformly bounded comparison of a non-orthogonal POD projection with an orthogonal POD projection.  These type of uniform boundedness conditions have been considered in \cite{chapelle2012galerkin,iliescu2014are,KeanSchneier20,LockeSingler20,singler2014new,xie2018numerical}, but they are not well understood.  We do not consider them further here; we leave them to be more fully explored elsewhere.

%%%%%%%%%%%%%%%%%%%%%%%%%%%%%%%%%%%%%

%\clearpage

\section{Numerical Results}
    \label{sec:numerical-results}

%In this section, we consider a test problem and suitable exact solutions that illustrate numerically the necessity of Assumption \ref{ass:LinftyTime}. 
In this section, we investigate numerically %the necessity of 
Assumption~\ref{ass:LinftyTime}.
Specifically, we consider the following questions:
(i) Is Assumption~\ref{ass:LinftyTime} satisfied?
(ii) Is the pointwise in time projection error optimal?
(iii) Is the pointwise in time ROM error optimal?
To investigate these questions numerically, we use the two counterexamples proposed in Sections~\ref{sec:cex-1}-\ref{sec:cex-2}: counterexample 1, which was defined in~\eqref{eqn:cex-1}, and counterexample 2, which was defined in~\eqref{eqn:cex-2}.
For each counterexample, we consider both the noDQ case (i.e., when the DQs are not used to construct the ROM basis; see Section~\ref{subsec:POD_no_DQs}) and the DQ case (i.e., when the DQs are used to construct the ROM basis; see Section~\ref{sec:POD_diffQ}).

Based on the theoretical results in Sections~\ref{sec:pointwise-error-noDQ} and \ref{subsubsec:VDQ}, we expect the noDQ case to (i) violate Assumption~\ref{ass:LinftyTime} (see~\eqref{eqn:counterexample_motivation_ineq}); (ii) yield suboptimal pointwise projection errors (see~\eqref{eqn:counterexample_exact_error2} in Proposition~\ref{prop:counterexample}); and (iii) yield suboptimal pointwise ROM errors (see~\eqref{eq:noDQ_error_estimates_L2basis}).
In contrast, based on the theoretical results in Sections~\ref{sec:pointwise-error-DQ} and \ref{subsubsec:VDQ}, we expect the DQ case to (i) fulfill Assumption~\ref{ass:LinftyTime} (see Theorem  \ref{thm:uniform_estimates}); (ii) yield optimal pointwise projection errors (see Theorem~\ref{thm:uniform_estimates}); and (iii) yield optimal pointwise ROM errors (see~\ref{eq:L2_L2_bound}). 
%\blue{Fill in missing details!}

%\blue{Rephrase next.}
In our numerical investigation, we use the one-dimensional heat equation \eqref{eqn:heat}, which was used in the theoretical development in Section~\ref{sec:ErrEst}.
For all the numerical experiments, we consider $\nu = 1$.
We note that the time step, $\Delta t$, plays an important role in our theoretical and numerical investigation.
Indeed, an $(N+1) = (T \Delta t^{-1} + 1)$ factor determines the suboptimality of the pointwise projection and ROM error bounds for the noDQ case (see~\eqref{eqn:counterexample_exact_error2} and~\eqref{eq:noDQ_error_estimates_L2basis}, respectively).  
Thus, in our numerical investigation it is desirable to consider as many $\Delta t$ values as possible.
We note, however, that the two counterexamples that we investigate restrict the $\Delta t$ values that we can consider.
The reason is that, while the two counterexamples yield ROM basis functions that are scaled versions of the snapshots (which is advantageous for the theoretical development), the treatment of their boundary conditions is somewhat delicate.
Indeed, both counterexamples vanish at $x=0$, but not at $x=1$.
To simplify the numerical treatment of the right boundary condition, we consider snapshots at $\Delta t$ values for which $k \, \Delta t$ is an integer.
This choice yields snapshots that vanish both at $x=0$ and at $x=1$, which allows for a straightforward ROM construction.
To summarize, in our numerical investigation we  strive to consider optimal $k$ values that are large enough to ensure a large number  of $\Delta t$ values (while satisfying the restriction $k \, \Delta t \in \mathbbm{N}$), and also low enough so that the numerical approximation is accurate.

%For any $\Delta t$ values, both counterexamples have zero boundary on the left whereas we do not have this convenient for the right boundary. To vanish the counterexamples on the right boundary, we have to put some restrictions between the $k$ and $\Delta t$ values, i.e., $k \Delta t$ multiplication should be integer. Since we want to investigate the numerical results for more time levels thus we aim to choose the $k$ value large as we can.
%\blue{Comment on how we deal with the nonhomogeneous boundary conditions.  And have we already defined the heat equation?} \textcolor{purple}{I tried to give details for first comment. I have moved the heat equation to first page.}
% \begin{align}
% \displaystyle u_t -\nu  u _{xx}= f ~~
% %(x,t) \in 
% \text{on } [0,1] \times[0,T] \, .  \label{eq:heat_eqn} 
% \end{align}
%and the two different exact solutions \eqref{eq:exact_soln1} and \eqref{eq:exact_soln2}.
% \begin{align}
% \displaystyle u_{exact1}(x,t) = sin ((kt+1) \pi x ) \label{eq:exact_soln1} 
% \end{align}
% \begin{align}
%   u_\mathrm{exact2}(x,t) = \frac{1}{\sqrt{2 \delta} } \left( e^{-\alpha(1+ t/\delta)} \right)^{1/2} \sin( (kt+1)\pi x ). \label{eq:exact_soln2}
%   \end{align}

\paragraph{\it{Snapshot Generation}}
%We note that 
Counterexamples 1 and 2 display a highly oscillatory behavior for the relatively large $k$ values chosen (i.e., $k=128$ and $k=100$, respectively).
Thus, to minimize the numerical error in generating the snapshots, we do not use a standard (e.g., FE) discretization.
%, i.e., a classical spatial and temporal discretization.
Instead,  to construct the snapshots, we use the analytical forms of counterexamples 1 and 2 given in~\eqref{eqn:cex-1} and~\eqref{eqn:cex-2}, respectively.
% To obtain the snapshots, we just evaluate the counterexample 1 \eqref{eqn:cex-1} with the parameters $ \Delta  h =1/4096$, $k = 128$, and varied $\Delta t$ values.
%Furthermore, we do not use any numerical method to obtain the snapshots because quite large $k$ value in \eqref{eqn:cex-1} yields numerical noise during the computation.
% We generate the FOM results by evaluating the exact solution \eqref{eqn:cex-1} on $[0,1] \times [0,1]$ with  spacestep $ \Delta  x =1/4096$, timestep size $\Delta t= 1/16$ and $k = 128$.

\paragraph{\it{ROM Construction}}
% To construct the ROM basis, we collect $N+1$  equally spaced snapshots on the time interval $[0,1]$.
To construct the ROM basis, we collect equally spaced snapshots on the time interval $[0,1]$ and $[0,0.2]$ for counterexamples 1 and 2, respectively.
Thus, the snapshot matrix $K$ is $(N+1)$-dimensional in the noDQ case, and $(2N+1)$-dimensional in the DQ case, as explained in Section~\ref{subsec:POD_no_DQs} and Section~\ref{sec:POD_diffQ}, respectively.
% To construct the snapshot matrix $K$,  in~\eqref{eqn:POD_corr_matrix} we use the standard Lagrange interpolant operator with respect to the finite element (FE) nodes to  interpolate the analytical solution of counterexample 1 given in~\eqref{eqn:cex-1}.
To construct 
%the snapshot matrix 
$K$,  in~\eqref{eqn:POD_corr_matrix} we use the standard Lagrange interpolant operator with respect to the FE nodes to  interpolate the analytical solution of counterexamples 1 and 2. 
%given in~\eqref{eqn:cex-1} and \eqref{eqn:cex-2}, respectively.
%\textcolor{purple}{We do not use FE method to compute the solutions of \eqref{eq:heat-eqn} (snapshots) but we use that method to obtain the FE operators that use during the ROM construction.}
Next, we use $K$ to build the ROM basis for the noDQ and DQ cases.
We emphasize that, although $K$ has different dimensions in the noDQ and DQ cases, to ensure a fair comparison, we use the same $r$ value in all the numerical experiments. 
%for  counterexamples 1 and 2. 
%We need to add more explanation in here to explain why we are using different approaches for counterexample 1 and 2 for choosing of $r$ value.
%In our numerical investigation, we run the test cases for the varied $\Delta t$ values, i.e., number of snapshots.
%We collect a total of $N+1$ snapshots from the time interval $[0,1]$ with $\Delta t = 1/N$ to create noDQ and DQ ROM basis functions. We use $N+1$ snapshots to create noDQ ROM basis functions while we use $N+1$ snapshots and their difference quotients to construct DQ ROM basis functions.
We construct the ROM operators by using the FE mass and stiffness matrices, which are obtained by using a linear FE spatial discretization with mesh size $ \Delta h=1/4096$. 
As ROM initial condition, we use the $L^2$  projection of the initial condition in the noDQ  case, and the Ritz projection of the initial condition in the DQ case.
% We use these ROM operators to build the ROM, and run it over the time interval $[0,1]$ with the Crank-Nicolson time discretization and the timestep $\Delta t = 1/N$. 
We use these ROM operators to build the ROM, and run it over the time interval $[0,T]$ with the Crank-Nicolson time discretization and the timestep $\Delta t = T/N$. 
%In numerical test cases, we use $N=16$ to observe how noDQ and DQ pointwise projection errors change for fixed $\Delta t$.

\subsection{Counterexample 1} 
    \label{sec:numerical-results-cex-1}

In this section, we consider counterexample 1, which was proposed in~\eqref{eqn:cex-1} of Section~\ref{sec:cex-1}.
In all the numerical experiments in this section, we consider $k=128$ in~\eqref{eqn:cex-1}. 
The numerical results are organized as follows:
In Section~\ref{sec:pointwise-projection-error}, for both the noDQ and the DQ cases, we investigate numerically whether (i) Assumption~\ref{ass:LinftyTime} holds; and (ii) the pointwise projection error is optimal. 
In Section~\ref{sec:pointwise-rom-error-cex-1}, for both the noDQ and the DQ cases, we investigate numerically whether the pointwise ROM errors are optimal. 

As explained in Section~\ref{sec:cex-1}, counterexample 1 was constructed to display the suboptimality of the pointwise projection and ROM bounds when $r=N$ and $t=t_N$.
Thus, in our numerical investigation we also consider $r=N$ and $t=t_N$.

\subsubsection{Pointwise Projection Error} 
    \label{sec:pointwise-projection-error}

In this section, we investigate numerically whether Assumption~\ref{ass:LinftyTime} holds.
To this end, we monitor the magnitude of the projection error~\eqref{eqn:projection-error}
%\blue{define the pointwise projection error earlier!}
%\textcolor{red}{No, we have not defined it.}
\begin{align}
    \displaystyle \Big\| \eta^{proj}(.,t_{n}) \Big\|_{L^2} = \Bigg\| u(.,t_{n})- \sum_{i=1}^{N} \Big( u(.,t_{n}), \varphi_i \Big)_{L^2} \varphi_i \Bigg\|_{L^2} , 
    \qquad n = 0, \ldots, N,
    \label{eqn:pointwise-projection-error-0}
\end{align}
at all the time instances, and check whether there are large variations in its magnitude.
%\medskip
Furthermore, for various $\Delta t$ values, we investigate numerically whether the projection error~\eqref{eqn:pointwise-projection-error-0} at the last time step is suboptimal (i.e., it has a suboptimal $\Delta t^{-1}$ factor). 
%compute the pointwise projection errors with and without DQ process to observe whether the Assumption~\ref{ass:LinftyTime} holds or not. Furthermore, we investigate how the noDQ and DQ last time projection errors scale. 
Specifically, as shown in \eqref{eqn:noDQ_pointwise_general_bound} for counterexample 1 in the noDQ case, the projection error at the last time step satisfies 
\begin{align}
    \displaystyle 
    \Big\| \eta^{proj}(.,t_{N}) \Big\|_{L^2}^2 
    = C_{proj}^{noDQ}  
    \sum_{i = N+1}^{N+1} \lambda_i^{noDQ} \Big \| \varphi_i \Big \|_{L^2}^2 \, , 
    \label{eqn:pointwise-projection-error-1}
\end{align}
where
\begin{eqnarray}
    C_{proj}^{noDQ}
    = T \, \Delta t^{-1} + 1
    = (N+1) \, .
    \label{eqn:pointwise-projection-error-2}
\end{eqnarray}
Moreover, as shown in \eqref{eqn:POD_DQs_pointwise_bound2} for counterexample 1 in the DQ case, the projection error at the last time step satisfies 
\begin{align}
    \displaystyle 
    \Big\| \eta^{proj}(.,t_{N}) \Big\|_{L^2}^2 
    \leq C_{proj}^{DQ}  
    \sum_{i = N+1}^{N+1} \lambda_i^{DQ} \Big \| \varphi_i \Big \|_{L^2}^2 \, , 
    \label{eqn:pointwise-projection-error-3}
\end{align}
where
\begin{eqnarray}
    C_{proj}^{DQ}
    = \mathcal{O}(1) .
    \label{eqn:pointwise-projection-error-4}
\end{eqnarray}
In this section, we investigate numerically the scalings~\eqref{eqn:pointwise-projection-error-1} and~\eqref{eqn:pointwise-projection-error-3}.

\paragraph{\it{noDQ Case}}
In Table~\ref{table:noDQ-proj-error-table}, for the noDQ case, we list the pointwise projection errors \eqref{eqn:pointwise-projection-error-0} at each time step. 
These results show that the pointwise projection error at the last time step is orders of magnitude higher than the pointwise projection error at the other time steps.
Thus, we conclude that, in the noDQ case, counterexample 1 violates Assumption~\ref{ass:LinftyTime}.
%We observe that the errors at the different time instances do not have the same order of magnitude. Especially, the order of magnitude at the last time error is much more larger than the rest. The numerical results shows that the (CEX1) without DQ process violates the Assumption~\ref{ass:LinftyTime}. 
\begin{table}[h!] 
	\centering
	\begin{tabular}{|c|c|| c|c|| c|c |} 
		\hline
		$n$ & $ ||\eta^{proj } (.,t_n)||_{L^2}$  & $n$ & $ ||\eta^{proj }(.,t_n) ||_{L^2}$ & $n$ & $ ||\eta^{proj } (.,t_n)||_{L^2}$  \\ 
		\hline\hline
		$  0$ & $2.79e-08$  &  $ 6 $  & $ 2.11e-08  $ & $12 $ & $0.00e+00$ \\
		\hline
		$  1$ & $2.24e-08 $  &  $ 7 $  &$ 0.00e+00 $
  & $13 $ & $ 1.49e-08$ \\
		\hline
		$  2$ & $ 2.69e-08 $  &  $ 8 $  & $1.67e-08  $ & $14 $ & $7.45e-09 $ \\
		\hline
		$  3$ & $ 7.45e-09 $  &  $ 9 $  &$ 1.05e-08  $ & $ 15$ & $1.67e-08$ \\
		\hline
		$  4$ & $ 1.49e-08  $  &  $ 10 $  &$ 2.11e-08  $ & $ 16 $ & $7.07e-01$ \\
		\hline
		$  5$ & $  1.83e-08 $  &  $ 11 $  &$ 1.05e-08  $ & $ $ & $ $ \\			
		\hline
	\end{tabular}
	\caption {
	%Heat equation \eqref{eq:heat-eqn},
	Counterexample 1~\eqref{eqn:cex-1}, $\Delta t = 1/16$, noDQ case: 
	Pointwise projection error~\eqref{eqn:pointwise-projection-error-0} at each time step. 
	%$ ||\eta^{proj }(.,t_n) ||_{L^2}, \  n = 0,1...,16$.
	\label{table:noDQ-proj-error-table}
	} 
\end{table}

%\bigskip

In Table~\ref{table:noDQ-proj-scaling-table}, we list the scaling factor~\eqref{eqn:pointwise-projection-error-1} for different $\Delta t$ values.
As expected from~\eqref{eqn:pointwise-projection-error-2}, these results show that the scaling factor is equal to $(N+1)$.
Thus, we conclude that, in the noDQ case, counterexample 1 yields suboptimal pointwise projection errors.

%Furthermore, since we have the largest pointwise projection error at the last time step for $\Delta t = 1/16$, (see Table~\ref{table:noDQ-proj-error-table}) , we aim to observe how the last time pointwise projection error behave as $\Delta t $ varies. In other words, we investigate whether the constant term in \eqref{eq:section611_scaling} is bounded or not as $\Delta t $ changes. In Table~\ref{table:noDQ-proj-scaling-table}, we compute and list the constant term $\mu_{exact1}$ for different $\Delta t$ values. By looking the numerical results, we observe that for each $\Delta t$ value, the constant term $\mu_{exact1}$ is represented as $(N+1)$. This observation confirms the equality \eqref{eqn:noDQ_pointwise_general_bound}, which proves that the last time projection error is not bounded. Tables~\ref{table:noDQ-proj-error-table} and \ref{table:noDQ-proj-scaling-table} conclude that (CEX1) without DQ case fails the Assumption~\ref{ass:LinftyTime}.
\begin{table}[h!] 
	\centering
	\begin{tabular}{|c | c| c| c| c| c| c| } 
		\hline
		$\Delta t$ &  $1/4$ & $1/8$ & $1/16$ & $1/32$ & $1/64$ & $1/128$  \\ 
		\hline\hline
% 	$C_{proj}^{noDQ}$ & $5.00e+00$ & $9.00e+00$ & $1.70e+01$ & $3.30e+01$ & $6.50e+01 $ & $1.29e+02$\\
	$C_{proj}^{noDQ}$ & $5.0e+00$ & $9.0e+00$ & $1.7e+01$ & $3.3e+01$ & $6.5e+01 $ & $1.3e+02$\\
		\hline																
	\end{tabular}
	\caption{
% 	Heat equation \eqref{eq:heat-eqn}, counterexample 1 \eqref{eqn:cex-1}, and noDQ case: compute $ \mu_{exact1}$ in \eqref{eq:section611_scaling} for different $\Delta t $ values.
	Counterexample 1~\eqref{eqn:cex-1}, noDQ case: 
	Scaling factor~\eqref{eqn:pointwise-projection-error-1} for different time step values. 
	}  \label{table:noDQ-proj-scaling-table}
\end{table}

\paragraph{\it{DQ Case}}
In Table~\ref{table:DQ-proj-error-table}, for the DQ case, we list the pointwise projection errors \eqref{eqn:pointwise-projection-error-0} at each time step. 
These results show that, in contrast with the noDQ case, the pointwise projection error at the last time step is of the same order of magnitude as the pointwise projection error at the other time steps.
% \blue{Birgul, why are all the errors the same?}
% \textcolor{purple}{They are slightly different since I list the two decimal, they seem to same.}
Thus, we conclude that, in the DQ case, counterexample 1 satisfies Assumption~\ref{ass:LinftyTime}.
%In Table~\ref{table:DQ-proj-error-table}, for the DQ case, we list the  pointwise projection errors \eqref{eqn:pointwise-projection-error-0} at each time step. Contrary to noDQ pointwise projection errors, we observe that the DQ pointwise projection errors in Table \ref{table:DQ-proj-error-table} have the same order of magnitude for all the time instances. In other words, there is no time instance at which the DQ pointwise projection error dominates the rest. Thus, we can conclude that, in the DQ case, counterexample 1 holds the Assumption~\ref{ass:LinftyTime}. 

% \begin{align}
% \displaystyle \Big\| \eta_{proj}^{exact1}(.,t_{N}) \Big\|_{L^2}^2 \leq \mathcal{C}_{exact1} \sum_{i = N+1}^{N+1} \lambda_i^{DQ} \Big \| \varphi_i \Big \|_{L^2}^2 \label{eq:section612_scaling}
% \end{align}

\begin{table}[h!] 
	\centering
	\begin{tabular}{|c|c|| c|c|| c|c |} 
		\hline
		$n$ & $ ||\eta^{proj } (.,t_n)||_{L^2}$  & $n$ & $ ||\eta^{proj }(.,t_n) ||_{L^2}$ & $n$ & $ ||\eta^{proj } (.,t_n)||_{L^2}$  \\ 
		\hline\hline
		$  0$ & $1.7144e-01$  &  $ 6 $  & $1.7144e-01$ & $12 $ & $1.7146e-01$ \\
		\hline
		$  1$ & $1.7144e-01$  &  $ 7 $  &   $1.7145e-01$
  & $13 $ & $1.7146e-01$ \\
		\hline
		$  2$ & $1.7144e-01$  &  $ 8 $  & $1.7145e-01$ & $14 $ & $1.7146e-01$ \\
		\hline
		$  3$ & $1.7144e-01$  &  $ 9 $  &$1.7145e-01$ & $ 15$ & $1.7146e-01$ \\
		\hline
		$  4$ & $1.7144e-01$  &  $ 10 $  &$1.7145e-01$ & $ 16 $ & $1.7147e-01$ \\
		\hline
		$  5$ & $1.7144e-01$  &  $ 11 $  &$1.7146e-01$  & $ $ & $ $ \\			
		\hline
	\end{tabular}
	\caption {
% 	Heat equation \eqref{eq:heat-eqn}, counterexample 1 \eqref{eqn:cex-1}, $\Delta t = 1/16$, and DQ case: compute $ ||\eta^{exact1}_{proj }(.,t_n) ||_{L^2}$ $\forall n = 0,1...,16$.
	Counterexample 1~\eqref{eqn:cex-1}, $\Delta t = 1/16$, DQ case: 
	Pointwise projection error~\eqref{eqn:pointwise-projection-error-0} at each time step. 
    \label{table:DQ-proj-error-table}
	}  
\end{table}

\bigskip

In Table~\ref{table:DQ-proj-scaling-table}, we list the scaling factor \eqref{eqn:pointwise-projection-error-3} for different time step values. As expected from~\eqref{eqn:pointwise-projection-error-4}, these results show that the scaling factor is bounded. Thus, we conclude that, in the DQ case, counterexample 1 yields optimal pointwise projection errors.

\begin{table}[h!] 
	\centering
	\begin{tabular}{|c | c| c| c| c| c| c| c| } 
		\hline
		$\Delta t$ &  $1/4$ & $1/8$ & $1/16$ & $1/32$ & $1/64$ & $1/128$  \\ 
		\hline\hline
% 	$||\eta_{proj}^{exact1}(.,t_N)||_{L^2}$ & $3.16e-01$ & $2.36e-01$ & $1.71e-01$ & $1.23e-01$ & $8.77e-02$ & $6.22e-02$ \\
% 	\hline
% 	$ \sqrt{\lambda^{DQ}_N$ & $2.36e-01$ & $1.71e-01$ & $1.23e-01$ & $8.77e-02$ & $6.22e-02$ & $4.41e-02$ \\
% 	\hline
% 	$\mathcal{C}^{DQ}_{proj}$ & $4.24e-01$ & $3.24e-01$ & $2.39e-01$ & $1.73e-01$ & $ 1.24e-01$ & $8.79e-02$ \\
% 	\hline
%	$\mathcal{C}^{DQ}_{proj}$ & $1.80e+00$ & $1.89e+00$ & $1.94e+00$ & $1.97e+00$ & $1.98e+00$ & $1.99e+00$ \\
	$\mathcal{C}^{DQ}_{proj}$ & $1.8e+00$ & $1.9e+00$ & $1.9e+00$ & $2.0e+00$ & $2.0e+00$ & $2.0e+00$ \\
	\hline	
	\end{tabular}
	\caption {
% 	Heat equation \eqref{eq:heat-eqn}, counterexample 1 \eqref{eqn:cex-1} and DQ case: compute $\mathcal{C}_{exact1}$ in \eqref{eq:section612_scaling} for different $\Delta t $ values.
% % 	$ ||\eta^{exact1}_{proj }(.,t_N) ||_{L^2}$ and POD truncation errors for varied $\Delta t $ values.
	Counterexample 1~\eqref{eqn:cex-1}, DQ case: 
	Scaling factor~\eqref{eqn:pointwise-projection-error-3} for different time step values. 
	}  \label{table:DQ-proj-scaling-table}
\end{table} 

%\bigskip

The numerical results in this section support the theoretical results in Section~\ref{sec:pointwise-projection-error-estimates}.
Specifically, counterexample 1 satisfies Assumption~\ref{ass:LinftyTime} in the DQ case, but not in the noDQ case.
Furthermore, the pointwise projection error at the last time step is optimal in the DQ case, and suboptimal in the noDQ case.
%From that section, we numerically prove that if we create the counterexample 1 that fails the Assumption~\ref{ass:LinftyTime}, by adding DQ part to process of pointwise projection error, the errors satisfy the Assumption~\ref{ass:LinftyTime} as predicted.

%\clearpage

\subsubsection{Pointwise ROM Error}             \label{sec:pointwise-rom-error-cex-1}

In this section, we investigate 
%numerically 
whether the pointwise ROM error is suboptimal. 

%monitor the error definitions in \eqref{eqn:pointwise-rom-error-1} and \eqref{eqn:pointwise-rom-error-2} for noDQ and DQ cases, respectively.

\paragraph{\it{noDQ Case}}

In the noDQ case, we investigate numerically the error estimate proved in \eqref{eq:noDQ_error_estimates_L2basis}:
\begin{align} \label{eqn:pointwise-rom-error-cex-1-1}
\max_{1 \leq k \leq N} \|e^{k}\|_{L^{2}}^{2}
= \mathcal{O}
% \left( (N+1) \, \sum_{i=r+1}^{N+1} \lambda_{i}^{noDQ} \|\varphi_{i}\|^{2}_{L^{2}} + \Delta t^4 + \sum_{i=r+1}^{N+1} \lambda_i^{noDQ} \| \nabla \varphi_i \|^2_{L^2}
\left( (N+1) \, \sum_{i=N+1}^{N+1} \lambda_{i}^{noDQ} \|\varphi_{i}\|^{2}_{L^{2}} + \Delta t^4 + \sum_{i=N+1}^{N+1} \lambda_i^{noDQ} \| \nabla \varphi_i \|^2_{L^2}
\right) .
\end{align}

% \begin{align} \label{eqn:pointwise-rom-error-cex-1-1}
% \max_{1 \leq k \leq N} \|e^{k}\|_{L^{2}}^{2}
% = \mathcal{O}
% \left( (N+1) \, \sum_{i=r+1}^{N+1} \lambda_{i}^{noDQ} \|\varphi_{i}\|^{2}_{L^{2}} + \Delta t^4 + \|{\phi^{0}_r\|_{L^2}^{2}  +  \sum_{i=r+1}^{N+1} \lambda_i^{noDQ} \| \nabla \varphi_i \|^2_{L^2}
% \right) .
% \end{align}

We note that, since the ROM initial condition is the $L^2$ projection of the initial condition, the term $\|\phi^{0}_r\|_{L^2}^{2}$ in~\eqref{eq:noDQ_error_estimates_L2basis} vanishes in~\eqref{eqn:pointwise-rom-error-cex-1-1}. 
As explained in Remark~\ref{remark:noDQ_error_estimates}, the error bound~\eqref{eqn:pointwise-rom-error-cex-1-1} is suboptimal with  respect to the time step due to the factor $(N+1) = (\Delta t^{-1} + 1)$ in the first term on the right-hand side.
To investigate numerically the suboptimality  of the error bound~\eqref{eqn:pointwise-rom-error-cex-1-1}, in Table~\ref{table:pointwise-rom-error-cex-1-1} we list the ratio
\begin{align} \label{eqn:pointwise-rom-error-cex-1-2}
    C_{rom}^{noDQ}
    =
    \left( \max_{1 \leq k \leq N} \|e^{k}\|_{L^{2}}^{2} \right) /
    & \left( (N+1) \, \sum_{i=N+1}^{N+1} \lambda_{i}^{noDQ} \|\varphi_{i}\|^{2}_{L^{2}} \right.
    \\
    & \left. \hspace*{0.6cm}
    + \, \Delta t^4 + \sum_{i=N+1}^{N+1} \lambda_i^{noDQ} \| \nabla \varphi_i \|^2_{L^2}
    \right) .
    \nonumber 
\end{align}
% \begin{align} 
%     C_{rom}^{noDQ}
%     =
%     \left( \max_{1 \leq k \leq N} \|e^{k}\|_{L^{2}}^{2} \right) /
%     \left( (N+1) \, \sum_{i=N+1}^{N+1} \lambda_{i}^{noDQ} \|\varphi_{i}\|^{2}_{L^{2}} + \Delta t^4 + \nor{\phi^{0}_r}{L^2}^{2} + \sum_{i=N+1}^{N+1} \lambda_i^{noDQ} \| \nabla \varphi_i \|^2_{L^2}
%     \right) .
%     \label{eqn:pointwise-rom-error-cex-1-2}
% \end{align}

The results in Table~\ref{table:pointwise-rom-error-cex-1-1} show that the ratio~\eqref{eqn:pointwise-rom-error-cex-1-2} is bounded from below.
Thus, we conclude that the pointwise ROM error in the noDQ case is suboptimal.
%\textcolor{red}{Shall we factor $(N+1)$ term in RHS of \eqref{eqn:pointwise-rom-error-1} ?}
\begin{table}[h!] 
	\centering
	\begin{tabular}{|c | c| c| c| c| c| c| } 
		\hline
		$\Delta t$ &  $1/4$ & $1/8$ & $1/16$ & $1/32$ & $1/64$ & $1/128$  \\ 
		\hline\hline
% 	$\mathcal{C}_{rom}^{noDQ}$ & $3.04e-05$ & $5.48e-05$ & $1.04e-04$ & $2.01e-04$ & $3.96e-04$ & $7.86e-04$   \\
% 		\hline
%	$\mathcal{C}_{rom}^{noDQ}$ & $2.97e-04$	& $1.82e-04$ & $1.04e-04$ & $2.01e-04$ & $7.59e-04$ & $7.86e-04$  \\
	$\mathcal{C}_{rom}^{noDQ}$ & $3.0e-04$	& $1.8e-04$ & $1.0e-04$ & $2.0e-04$ & $7.6e-04$ & $7.9e-04$  \\
	\hline	
% 	$\mathcal{C}_{rom}^{noDQ}$ & $3.04e-05$ & $ 5.48e-05$ & $1.04e-04$ & $2.01e-04$ & $3.96e-04$ & $7.86e-04$  \\
% 	\hline		
		
	\end{tabular}
	\caption {
	Counterexample 1~\eqref{eqn:cex-1}, noDQ case: 
	Ratio~\eqref{eqn:pointwise-rom-error-cex-1-2} for different time step values. %\textcolor{purple}{2nd row contains the updated results.}
	\label{table:pointwise-rom-error-cex-1-1}
	}  
\end{table} 

To investigate the sensitivity of our numerical results with respect to $k$ (i.e., the level of oscillations in counterexample 1), in Table~\ref{table:pointwise-rom-error-cex-1-2} we list the ratio~\eqref{eqn:pointwise-rom-error-cex-1-2} for $k=8$.
The results in Table~\ref{table:pointwise-rom-error-cex-1-2} confirm the results in Table~\ref{table:pointwise-rom-error-cex-1-1}, i.e., the pointwise ROM error in the noDQ case is suboptimal.

\begin{table}[h!] 
	\centering
	\begin{tabular}{|c | c| c| c|  } 
		\hline
		$\Delta t$ &  $1/2$ & $1/4$  & $1/8$ \\ 
		\hline\hline
	$\mathcal{C}_{rom}^{noDQ}$ & $3.75e-03$ & $6.371e-03$ 
	& $1.13-02$ \\
	\hline	
	\end{tabular}
	\caption {
	Counterexample 1~\eqref{eqn:cex-1}, $k=8$, noDQ case: 
	Ratio~\eqref{eqn:pointwise-rom-error-cex-1-2} for different time step values. 
	\label{table:pointwise-rom-error-cex-1-2}
	}  
\end{table}

\paragraph{\it{DQ Case}}
In the DQ case, we investigate numerically the error estimate proved in \eqref{eq:L2_L2_bound}:
% How do we scale DQ ROM error? Does it optimal?
% \begin{align} 
%     \label{eqn:pointwise-rom-error-cex-1-3}
%     \|e^{N}\|_{L^{2}}^{2}
%     = \mathcal{O}
%     \left(\sum_{i=N+1}^{N+1} \lambda_{i}^{DQ}\|\varphi_{i} -R_{r}(\varphi_{i})\|^{2}_{L^{2}} 
%     + \Delta t^{4}
%     \right) .
% \end{align}

\begin{align} 
    \label{eqn:pointwise-rom-error-cex-1-3}
\max_{1 \leq k \leq N} \|e^{k}\|_{L^{2}}^{2} =
\mathcal{O}
% \left( \sum_{i=r+1}^{N+1} \lambda_{i}^{DQ}\|\varphi_{i} -R_{r}(\varphi_{i})\|^{2}_{L^{2}}  + \Delta t^{4}  \right),
\left( \sum_{i=N+1}^{N+1} \lambda_{i}^{DQ}\|\varphi_{i} -R_{r}(\varphi_{i})\|^{2}_{L^{2}}  + \Delta t^{4}  \right),
\end{align}

We note that the error bound~\eqref{eqn:pointwise-rom-error-cex-1-3} is optimal.
In Table~\ref{table:pointwise-rom-error-cex-1-3}, we list the ratio
% \begin{align} 
%     C_{rom}^{DQ}
%     =
%     \left( \|e^{N}\|_{L^{2}}^{2} \right) / 
%     \left(\sum_{i=N+1}^{N+1} \lambda_{i}^{DQ}\|\varphi_{i} -R_{r}(\varphi_{i})\|^{2}_{L^{2}} 
%     + \Delta t^{4}
%     \right) .
%     \label{eqn:pointwise-rom-error-cex-1-4}
% \end{align}

\begin{align} 
    C_{rom}^{DQ}
    =
    \left( \max_{1 \leq k \leq N} \|e^{k}\|_{L^{2}}^{2} \right) / 
    % \left(\sum_{i=r+1}^{N+1} \lambda_{i}^{DQ}\|\varphi_{i} -R_{r}(\varphi_{i})\|^{2}_{L^{2}} 
    % + \Delta t^{4}
    % \right) .
    \left(\sum_{i=N+1}^{N+1} \lambda_{i}^{DQ}\|\varphi_{i} -R_{r}(\varphi_{i})\|^{2}_{L^{2}} 
    + \Delta t^{4}
    \right) .
    \label{eqn:pointwise-rom-error-cex-1-4}
\end{align}

The results in Table~\ref{table:pointwise-rom-error-cex-1-3} show that the ratio~\eqref{eqn:pointwise-rom-error-cex-1-4}, while increasing, seems to be bounded, as predicted by~\eqref{eqn:pointwise-rom-error-cex-1-3}.

\begin{table}[h!] 
	\centering
	\begin{tabular}{|c | c| c| c| c| c| c| } 
		\hline
		$\Delta t$ &  $1/4$ & $1/8$ & $1/16$ & $1/32$ & $1/64$ & $1/128$  \\ 
		\hline
% 		\hline
% 	$||e^N||_{L^2}$ & $7.03e-01$ & $6.97e-01$ & $6.88e-01$ & $6.60e-01$ & $5.82e-01$ & $4.50e-01 $   \\
	\hline
% 	$C_{rom}^{DQ}$ & $7.78e-02$ & $1.34e-01$ & $2.03e-01$ & $3.52e-01$ & $5.29e-01$ & $8.73e-01$ \\
	$C_{rom}^{DQ}$ & $7.8e-02$ & $1.3e-01$ & $2.0e-01$ & $3.5e-01$ & $5.3e-01$ & $8.7e-01$ \\
		\hline																
	\end{tabular}
	\caption {
	Counterexample 1~\eqref{eqn:cex-1}, DQ case: 
	Ratio~\eqref{eqn:pointwise-rom-error-cex-1-4} for different time step values. 
	  \label{table:pointwise-rom-error-cex-1-3}
	}
\end{table}

% \begin{table}[h!] 
% 	\centering
% 	\begin{tabular}{|c | c| c| c| c| } 
% 		\hline
% 		$\Delta t$ &  $1/2$ & $1/4$ & $1/8$ & $1/16$   \\ 
% 		\hline\hline
% 	$||e^N||_{L^2}$ & $7.02e-01$ & $7.02e-01$ & $5.92e-01$ & $4.69e-01$\\
% 	\hline
% 	$\mathcal{C}_{rom}$ & $3.89e-01$ & $3.99e-01$ & $4.76e-01$ & $1.11e+00$ \\
% 		\hline										
% 	\end{tabular}
% 	\caption {Heat equation \eqref{eq:heat equation}, exact solution \eqref{eq:exact-soln}, $k=16$, $\Delta h = 1/1024$, and $\nu = 1$, $\mathcal{P}_1$ FE basis, Crank-Nicolson time evaluation, and DQ case: compute $ \mathcal{C}_{rom}$ for different $\Delta t $ values.
% 	\textcolor{blue}{$C_{rom}$ doesn't seem to increase that much.}
% 	}  \label{table:section522-scaling-k16}
% \end{table} 

The increase of $C_{rom}^{DQ}$ in Table~\ref{table:pointwise-rom-error-cex-1-3} is due to the highly oscillatory character of counterexample 1 in~\eqref{eqn:cex-1}, which makes the ROM simulation in the DQ case challenging.
To alleviate the highly oscillatory behavior of counterexample 1, we keep all the parameters unchanged and choose a lower $k$ value (i.e., $k=8$) in~\eqref{eqn:cex-1}, which yields a solution with fewer oscillations.
In Table~\ref{table:pointwise-rom-error-cex-1-4}, we list the ratio~\eqref{eqn:pointwise-rom-error-cex-1-4} for $k=8$.
The results in Table~\ref{table:pointwise-rom-error-cex-1-4} show that the ratio~\eqref{eqn:pointwise-rom-error-cex-1-4} is bounded, as predicted by~\eqref{eqn:pointwise-rom-error-cex-1-3}.

\begin{table}[h!] 
	\centering
	\begin{tabular}{|c | c| c| c| } 
		\hline
		$\Delta t$ &  $1/2$ & $1/4$ & $1/8$   \\ 
		\hline\hline
% 	$||e^N||_{L^2}$ & $7.01e-01$ & $6.49e-01 $ & $4.86e-01$ \\
% 	\hline
	$\mathcal{C}_{rom}^{DQ}$ & $4.73e-01$ & $5.92e-01$ & $2.55e-01$  \\
		\hline										
	\end{tabular}
	\caption {
% 	Heat equation \eqref{eq:heat equation}, exact solution \eqref{eq:exact-soln}, $k=8$, $\Delta h = 1/1024$, and $\nu = 1$, $\mathcal{P}_1$ FE basis, Crank-Nicolson time evaluation, and DQ case: compute $ \mathcal{C}_{rom}$ for different $\Delta t $ values.
   Counterexample 1~\eqref{eqn:cex-1}, $k=8$, and DQ case: 
	Ratio~\eqref{eqn:pointwise-rom-error-cex-1-4} for different time step values. 
	}  \label{table:pointwise-rom-error-cex-1-4}
\end{table} 

%\smallskip

The numerical results in this section support the theoretical results in Section~\ref{sec:ErrEst}.
Specifically, for counterexample 1, the pointwise ROM error is optimal in the DQ case, and suboptimal in the noDQ case.

%\clearpage
\subsection{Counterexample 2} 
    \label{sec:numerical-results-cex-2}
In this section, we consider counterexample 2, which was proposed in equation~\eqref{eqn:cex-2} of Section~\ref{sec:cex-2}.
In all the numerical experiments in this section, we consider $k=100$, $\delta = 0.01$, and $\alpha = 1$ in~\eqref{eqn:cex-2}. 
The numerical results are organized as follows:
In Section~\ref{sec:pointwise-projection-error-cex-2}, for both the noDQ and the DQ cases, we investigate numerically whether (i) Assumption~\ref{ass:LinftyTime} holds; and (ii) the pointwise projection error is optimal. 
In Section~\ref{sec:pointwise-rom-error-cex-2}, for both the noDQ and the DQ cases, we investigate numerically whether the pointwise ROM errors are optimal. 

\smallskip

As explained in Section~\ref{sec:cex-2}, counterexample 2 was constructed to display the suboptimality of the pointwise projection and ROM error bounds for any $r$ values.
%Counterexample 2 was constructed to display the suboptimal bound with $r=N$ and $t=t_N$ for pointwise projection whereas ROM yields suboptimal bound independently choosing of $r$ and $t_i$ values.
% In our numerical investigation, we consider $r=N$ (and $t=t_N$) for the pointwise projection error (in order to assess the magnitude of the pointwise projection error at all the time instances),
% and general $r$ values for the pointwise ROM error.
In our numerical investigation, we consider general $r$ and $t=t_k$ values for both the pointwise projection error and the pointwise ROM error.

\subsubsection{Pointwise Projection Error} 
\label{sec:pointwise-projection-error-cex-2}
In this section, we investigate numerically whether Assumption~\ref{ass:LinftyTime} holds.
To this end, for various $\Delta t$ values, we investigate numerically whether the projection error~\eqref{eqn:pointwise-projection-error-5} 
% at the last time step is suboptimal (i.e., it has a suboptimal $\Delta t^{-1}$ factor). 
at various time instances is suboptimal.
%we monitor the magnitude of the projection error~\eqref{eqn:projection-error}
\begin{align}
    \displaystyle \Big\| \eta^{proj}(.,t_{r}) \Big\|_{L^2} = \Bigg\| u(.,t_{r})- \sum_{i=1}^{r} \Big( u(.,t_{r}), \varphi_i \Big)_{L^2} \varphi_i \Bigg\|_{L^2} , 
    \qquad r = 1, \ldots, N,
    \label{eqn:pointwise-projection-error-5}
\end{align}
% at all the time instances, and check whether there are large variations in its magnitude.
% \medskip
Specifically, as shown in \eqref{eqn:counterexample_bound_below} in  Proposition~\ref{prop:counterexample} for counterexample 2 in the noDQ case, for fixed $r$ values, the projection error at $t=t_r$ satisfies 
%\eqref{eqn:pointwise-projection-error-8} and \eqref{eqn:pointwise-projection-error-9}. 
% Specifically, as shown in \eqref{eqn:noDQ_pointwise_general_bound} for counterexample 2 in the noDQ case, the projection error at the last time step satisfies \eqref{eqn:pointwise-projection-error-1} and \eqref{eqn:pointwise-projection-error-2}.
% \begin{align}
%     \displaystyle 
%     \Big\| \eta^{proj}(.,t_{N}) \Big\|_{L^2}^2 
%     = C_{proj}^{noDQ}  
%     \sum_{i = N+1}^{N+1} \lambda_i^{noDQ} \Big \| \varphi_i \Big \|_{L^2}^2 \, , 
%     \label{eqn:pointwise-projection-error-1}
% \end{align}
% where
% \begin{eqnarray}
%     C_{proj}^{noDQ}
%     = T \, \Delta t^{-1} + 1
%     = (N+1) \, .
%     \label{eqn:pointwise-projection-error-2}
% \end{eqnarray}
%\medskip
% Furthermore, as shown in \eqref{eqn:POD_DQs_pointwise_bound2} for counterexample 2 in the DQ case, the projection error at the last time step satisfies \eqref{eqn:pointwise-projection-error-3} and \eqref{eqn:pointwise-projection-error-4}.
% \begin{align}
%     \displaystyle 
%     \Big\| \eta^{proj}(.,t_{N}) \Big\|_{L^2}^2 
%     \leq C_{proj}^{DQ}  
%     \sum_{i = N+1}^{N+1} \lambda_i^{DQ} \Big \| \varphi_i \Big \|_{L^2}^2 \, , 
%     \label{eqn:pointwise-projection-error-3}
% \end{align}
% where
% \begin{eqnarray}
%     C_{proj}^{DQ}
%     = \mathcal{O}(1) .
%     \label{eqn:pointwise-projection-error-4}
% \end{eqnarray}
%\medskip
% In this section, we investigate numerically the scalings~\eqref{eqn:pointwise-projection-error-2} and~\eqref{eqn:pointwise-projection-error-4}. 
\begin{align}
    \displaystyle 
    \Big\| \eta^{proj}(.,t_{r}) \Big\|_{L^2}^2 
    = C_{proj}^{noDQ} \, (N+1) \,
    \sum_{i = r+1}^{N+1} \lambda_i^{noDQ} \Big \| \varphi_i \Big \|_{L^2}^2 \, , 
   \label{eqn:pointwise-projection-error-6}
\end{align}
where
%from \eqref{eqn:counterexample_bound_below} in the Proposition~\ref{prop:counterexample}, we expect that
\begin{eqnarray}
    C_{proj}^{noDQ}
\geq  \frac{\min\{ 1, \gamma \}}{2} \, .
   \label{eqn:pointwise-projection-error-7}
\end{eqnarray}

Moreover, as shown in \eqref{eqn:POD_DQs_pointwise_bound2} for counterexample 2 in the DQ case, the projection error at various time instances satisfies 
%\eqref{eqn:pointwise-projection-error-10} and \eqref{eqn:pointwise-projection-error-11}. 
\begin{align}
    % \displaystyle 
    % \Big\| \eta^{proj}(.,t_{j}) \Big\|_{L^2}^2 
    % \leq C_{proj}^{DQ}  
    % \sum_{i = j+1}^{N+1} \lambda_i^{DQ} \Big \| \varphi_i \Big \|_{L^2}^2 \, , 
    %  \label{eqn:pointwise-projection-error-10}
    \displaystyle 
	\max_{0 \leq k \leq N} \Bigg\| u(.,t_{k})- \sum_{i=1}^{r} \Big( u(.,t_{k}), \varphi_i \Big)_{L^2} \varphi_i \Bigg\|^2_{L^2} &\leq C_{proj}^{DQ}  \sum_{i = r+1}^d \lambda_i^\mathrm{DQ} \| \varphi_i \|_{L^2}, \label{eqn:pointwise-projection-error-8}
\end{align}
where
% \begin{eqnarray}
%     C_{proj}^{DQ}
%     = \mathcal{O}(1) .
%     \label{eqn:pointwise-projection-error-11}
% \end{eqnarray}
$C_{proj}^{DQ}$ is bounded from above.
In this section, we investigate numerically the scalings~\eqref{eqn:pointwise-projection-error-6}--\eqref{eqn:pointwise-projection-error-7} and~\eqref{eqn:pointwise-projection-error-8}.

%\bigskip

\paragraph{\it{noDQ Case}}
% In Table~\ref{table:noDQ-proj-error-table-cex-2}, for the noDQ case, we list the pointwise projection errors \eqref{eqn:pointwise-projection-error-0} at each time step. 
% These results show that the pointwise projection error at the last time step is orders of magnitude higher than the pointwise projection error at the other time steps.
% Thus, we conclude that, in the noDQ case, counterexample 2 violates Assumption~\ref{ass:LinftyTime}.

% In Table~\ref{table:noDQ-proj-scaling-table-cex-2}, we list the scaling factor~\eqref{eqn:pointwise-projection-error-2} for different time step values.
% As expected, these results show that the scaling factor is equal to $(N+1)$.
% Thus, we conclude that, in the noDQ case, counterexample 2 yields suboptimal pointwise projection errors.

In Table~\ref{table:noDQ-proj-scaling-table-cex-2}, for $r=4$, we list the scaling factor $\mathcal{C}_{proj}^{DQ}$ in~\eqref{eqn:pointwise-projection-error-6} for different time step values. 
As expected from~\eqref{eqn:pointwise-projection-error-7}, these results show that the scaling factor  is bounded from below.
Thus, we conclude that, in the noDQ case, counterexample 2 yields suboptimal pointwise projection errors.

\begin{table}[h!] 
	\centering
	\begin{tabular}{| c| c| c| c| c|  } 
		\hline
		$\Delta t$  & $0.05$ & $0.04$ & $0.02$ & $0.01$  \\ 
		\hline\hline
	$\mathcal{C}_{proj}^{noDQ}$ &  $1.00e+00$ & $9.82e-01$ & $8.65e-01$ & $6.32e-01$ \\
		\hline								
	\end{tabular}
	\caption {
		Counterexample 2~\eqref{eqn:cex-2}, $r= 4$, and noDQ case: 
	Scaling factor~\eqref{eqn:pointwise-projection-error-6} for different time step values. %\red{Samuele: Should the column $\Delta t=0.1$ be removed?} \purple{Birgul: Done}. 
	} 
	\label{table:noDQ-proj-scaling-table-cex-2}
\end{table} 

%\bigskip

\paragraph{\it{DQ Case}}
% In Table~\ref{table:DQ-proj-error-table-cex-2}, for the DQ case, we list the pointwise projection errors \eqref{eqn:pointwise-projection-error-0} at each time step. 
% These results show that, in contrast with the noDQ case, the pointwise projection error at the last time step is on the same order of magnitude as the pointwise projection error at the other time steps.

% In Table~\ref{table:DQ-proj-scaling-table-cex-10}, we list the scaling factor \eqref{eqn:pointwise-projection-error-4} for different time step values. As expected, these results show that the scaling factor is bounded. Thus, we can conclude that, in the DQ case, counterexample 2 yields optimal pointwise projection errors.

In Table~\ref{table:DQ-proj-scaling-table-cex-2}, for $r=4$, we list the scaling factor \eqref{eqn:pointwise-projection-error-8} for different time step values. As expected, these results show that the scaling factor is bounded from above. Thus, we conclude that, in the DQ case, counterexample 2 yields optimal pointwise projection errors.

\begin{table}[h!] 
	\centering
	\begin{tabular}{| c| c| c| c| c|  } 
		\hline
		$\Delta t$  & $ 0.05$  & $ 0.04 $ & $ 0.02 $  & $ 0.01 $ \\
		\hline\hline
	$\mathcal{C}_{proj}^{DQ}$  & $1.83e+00$ & $1.76e-02$ & $8.32e-03$ & $3.84e-03$ \\	
	\hline
	\end{tabular}
	\caption {Counterexample 2~\eqref{eqn:cex-2}, $r=4$, and DQ case: 
	Scaling factor~\eqref{eqn:pointwise-projection-error-8} for different time step values.} \label{table:DQ-proj-scaling-table-cex-2}
\end{table} 

%\bigskip

The numerical results in this section support the theoretical results in Section~\ref{sec:pointwise-projection-error-estimates}.
Specifically, for a generic $r$ value, counterexample 2 satisfies Assumption~\ref{ass:LinftyTime} in the DQ case, but not in the noDQ case.
Furthermore, the pointwise projection error 
%at the last time step 
is optimal in the DQ case, and suboptimal in the noDQ case.

\subsubsection{Pointwise ROM Error}\label{sec:pointwise-rom-error-cex-2} 
In this section, we investigate 
%numerically 
whether the pointwise ROM error is suboptimal. 
We note that the time evolution of the analytical solution in counterexample 2 (which is displayed in Figure~\ref{fig:cex-2-plot}) 
%\red{Samuele: Cannot find the figure?!} \purple{Birgul: It is added}) 
prompted us to make the following parameter choices in the numerical investigation of the pointwise ROM error.
Since the magnitude of the analytical solution is significant on the time interval $[0,0.04]$ and almost negligible on the time interval $[0.04,0.2]$, we decided to compute the pointwise ROM errors for both the noDQ and the DQ cases on the time interval $[0, 0.05]$.
Furthermore, since the DQ ROM basis functions with large indices are very oscillatory, we decided to use low $r$ values in order to avoid numerical instabilities. 

 \begin{figure}[h!] 
 \begin{center}
 \includegraphics[width=0.48\textwidth,height=0.46\textwidth]{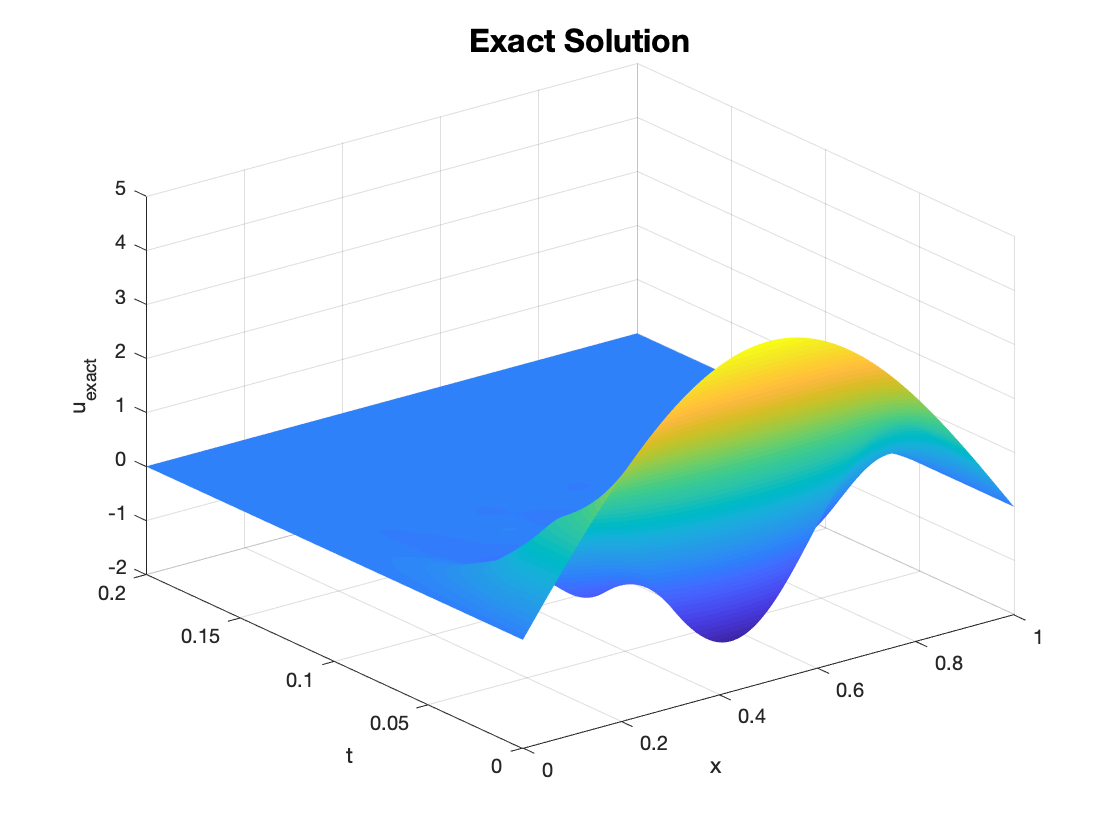}
  \caption{
Counterexample 2~\eqref{eqn:cex-2}, FOM plot:  
%$\delta = 0.01$, $k=100$, $\alpha = 1$, 
$h = 1/4096$ and $\Delta t = 0.02$.
  } \label{fig:cex-2-plot}
 \end{center} 
\end{figure}

%Before moving on the numerical results for both noDQ and DQ cases, let's look at a bit nature of the counterexample 2. The Figure~\ref{fig:cex-2-plot} which is the plot of the counterexample 2, seems to display some dynamics on [0,0.04], and maybe some relative small magnitude dynamics on [0.04, 0.1]. After that, the solution dies out. Thus,  we will compute the ratios \eqref{eqn:pointwise-rom-error-cex-2-2} in noDQ case and \eqref{eqn:pointwise-rom-error-cex-2-4} in DQ case over the shorter time interval where main dynamics display.
% \medskip
% Furthermore, we observe that DQ ROM basis with large indices are very oscillatory and that causes some numerical instabilities. Thus, we use small $r$ values when we compute the ratio \eqref{eqn:pointwise-rom-error-cex-2-4}, which is independent of $r$ value. To make a fair comparison, we use same idea to compute the ratio \eqref{eqn:pointwise-rom-error-cex-2-2} as well.

\paragraph{\it{noDQ Case}}

In the noDQ case, we investigate numerically the error estimate proved in \eqref{eq:noDQ_error_estimates_L2basis}:
\begin{align} \label{eqn:pointwise-rom-error-cex-2-1}
\max_{1 \leq k \leq N} \|e^{k}\|_{L^{2}}^{2}
= \mathcal{O}
\left( (N+1) \, \sum_{i=r+1}^{N+1} \lambda_{i}^{noDQ} \|\varphi_{i}\|^{2}_{L^{2}} + \Delta t^4 + \sum_{i=r+1}^{N+1} \lambda_i^{noDQ} \| \nabla \varphi_i \|^2_{L^2}
\right) .
\end{align}

We note that, since the ROM initial condition is the $L^2$ projection of the initial condition, the term $\|\phi^{0}_r\|_{L^2}^{2}$ in~\eqref{eq:noDQ_error_estimates_L2basis} vanishes in~\eqref{eqn:pointwise-rom-error-cex-2-1}. 
As explained in Remark~\ref{remark:noDQ_error_estimates}, the error bound~\eqref{eqn:pointwise-rom-error-cex-2-1} is suboptimal with  respect to the time step due to the factor $(N+1) = (T \Delta t^{-1} + 1)$ in the first term on the right-hand side.
To investigate numerically the suboptimality  of the error bound~\eqref{eqn:pointwise-rom-error-cex-2-1}, in Table
%~\ref{table:pointwise-rom-error-cex-2-1} and 
\ref{table:pointwise-rom-error-cex-2-2} we list the ratio \eqref{eqn:pointwise-rom-error-cex-2-2} for fixed $\Delta t$ values and various $r$ values. 
The ratios in Table
%~\ref{table:pointwise-rom-error-cex-2-1} and 
\ref{table:pointwise-rom-error-cex-2-2} are bounded from below. 
%and the variation in \eqref{eqn:pointwise-rom-error-cex-2-2} is getting less for small $\Delta t$ value.
Thus, we conclude that the pointwise ROM error in the noDQ case is suboptimal.

\begin{align}     \label{eqn:pointwise-rom-error-cex-2-2}
    C_{rom}^{noDQ}
    =
    \left( \max_{1 \leq k \leq N} \|e^{k}\|_{L^{2}}^{2} \right) /
    & \left( (N+1) \, \sum_{i=r+1}^{N+1} \lambda_{i}^{noDQ} \|\varphi_{i}\|^{2}_{L^{2}}
    \right. \\
    & \hspace*{0.6cm} 
    \left. 
    + \, \Delta t^4 + \sum_{i=r+1}^{N+1} \lambda_i^{noDQ} \| \nabla \varphi_i \|^2_{L^2}
    \right) .
    \nonumber 
\end{align}

% \begin{table}[h!] 
% 	\centering
% 	\begin{tabular}{|c | c| c| c|  } 
% 		\hline
% 	$r$ & $ 1 $ & $2$  & $ 3 $     \\
% 		\hline\hline
% 	$\mathcal{C}_{rom}^{noDQ}$ & $7.91e-02$ & $9.24e-01$ & $3.96e+00$  
% \\
% 	\hline												
% 	\end{tabular}
% 	\caption {
%     Counterexample 2~\eqref{eqn:cex-2} and noDQ case: 
% 	Ratio~\eqref{eqn:pointwise-rom-error-cex-2-2} for shorter time interval $[0,0.04]$, fixed time step $\Delta t = 0.02$, and different $r$ values.
% 	} \label{table:pointwise-rom-error-cex-2-1}
% \end{table} 

\begin{table}[h!] 
	\centering
	\begin{tabular}{|c | c| c| c| c| c| c|  } 
		\hline
	$r$ & $ 1 $ & $2$  & $ 3 $ & $ 4 $ & $5$ & $6$  \\
		\hline\hline
% 	$\mathcal{C}_{rom}^{noDQ}$ & $1.69e-01$ & $9.83e-02$ & $1.13e-01$ & $2.18e-01$ &
% 	$4.40e-01$ & $9.20e-01$ \\
	$\mathcal{C}_{rom}^{noDQ}$ & $1.7e-01$ & $9.8e-02$ & $1.1e-01$ & $2.2e-01$ &
	$4.4e-01$ & $9.2e-01$ \\
	\hline																
	\end{tabular}
	\caption {
    Counterexample 2~\eqref{eqn:cex-2} and noDQ case: 
	Ratio~\eqref{eqn:pointwise-rom-error-cex-2-2} for %shorter time interval $[0,0.05]$, 
	fixed time step $\Delta t = 0.01$ and different $r$ values. 
	} \label{table:pointwise-rom-error-cex-2-2}
\end{table}

%\medskip

\paragraph{\it{DQ Case}}
In the DQ case, we investigate numerically the error estimate proved in \eqref{eq:L2_L2_bound}:

\begin{align} 
    \label{eqn:pointwise-rom-error-cex-2-3}
\max_{1 \leq k \leq N} \|e^{k}\|_{L^{2}}^{2} =
\mathcal{O}
\left( \sum_{i=r+1}^{N+1} \lambda_{i}^{DQ}\|\varphi_{i} -R_{r}(\varphi_{i})\|^{2}_{L^{2}}  + \Delta t^{4}  \right).
\end{align}

To investigate numerically the suboptimality  of the error bound~\eqref{eqn:pointwise-rom-error-cex-2-3}, in Table
%~\ref{table:pointwise-rom-error-cex-2-3} and 
~\ref{table:pointwise-rom-error-cex-2-4} we list the ratio \eqref{eqn:pointwise-rom-error-cex-2-2} for fixed $\Delta t$ values and various $r$ values. 
The ratios in Table
%s~\ref{table:pointwise-rom-error-cex-2-3} and 
~\ref{table:pointwise-rom-error-cex-2-4} are bounded. %from below. 
%and the variation in \eqref{eqn:pointwise-rom-error-cex-2-4} is getting less for small $\Delta t$ value.
Thus, we conclude that the pointwise ROM error in the DQ case is optimal.

\begin{align} 
    C_{rom}^{DQ}
    =
    \left( \max_{1 \leq k \leq N} \|e^{k}\|_{L^{2}}^{2} \right) / 
    \left(\sum_{i=r+1}^{N+1} \lambda_{i}^{DQ}\|\varphi_{i} -R_{r}(\varphi_{i})\|^{2}_{L^{2}} 
    + \Delta t^{4}
    \right) .
    \label{eqn:pointwise-rom-error-cex-2-4}
\end{align}

% \begin{table}[h!] 
% 	\centering
% 	\begin{tabular}{|c | c| c| c| c|  } 
% 		\hline
% 	$r$ & $ 1 $ & $2$  & $ 3 $   \\
% 		\hline\hline
% 	$\mathcal{C}_{rom}^{DQ}$ & $8.88e-03$ & $3.02e-02$ & $3.40e-01$   \\
% 	\hline	
% 	\end{tabular}
% 	\caption {Counterexample 2~\eqref{eqn:cex-2} and DQ case: 
% 	Ratio~\eqref{eqn:pointwise-rom-error-cex-2-4} for the shorter time interval $[0,0.04]$ and fixed time step $\Delta t = 0.02$ and different $r$ values. 
% 	} \label{table:pointwise-rom-error-cex-2-3}	
% \end{table}

\begin{table}[h!] 
	\centering
	\begin{tabular}{|c | c| c| c| c| c| c| } 
		\hline
	$r$ & $ 1 $ & $2$  & $ 3 $ & $ 4 $  & $ 5 $ & $6$ \\
		\hline\hline
% 	$ \mathcal{C}_{rom}^{DQ}$ & $2.88e-03$ & $3.96e-03$ & $4.94e-03$ & $5.69e-03$ & $1.00e-02$ & $2.89e-02$ \\	
	$ \mathcal{C}_{rom}^{DQ}$ & $2.9e-03$ & $4.0e-03$ & $4.9e-03$ & $5.7e-03$ & $1.0e-02$ & $2.9e-02$ \\	
	\hline
	\end{tabular}
	\caption {Counterexample 2~\eqref{eqn:cex-2} and DQ case: 
	Ratio~\eqref{eqn:pointwise-rom-error-cex-2-4} for %the shorter time interval $[0,0.05]$ and 
	fixed time step $\Delta t = 0.01$ and different $r$ values. 
	} \label{table:pointwise-rom-error-cex-2-4}	
\end{table}

The numerical results in this section support the theoretical results in Section~\ref{sec:ErrEst}.
Specifically, for counterexample 2, the pointwise ROM error is optimal in the DQ case, and suboptimal in the noDQ case.

% 	\begin{tabular}{|c | c| c| c| c| c|c| } 
% 		\hline
% 	$r$ & $ 1 $ & $2$  & $ 3 $ & $ 4 $ & $5$ & $6$   \\
% 		\hline\hline
% 	$||e_{noDQ}||_{\infty}$ & $7.36e-01$ & $7.01e-01$ & $7.01e-01$ & $6.99e-01$ & $6.96e-01$ & $6.75e-01$  \\
% 	\hline
% 	$||e_{DQ}||_{\infty}$ & $7.14e-01$ & $7.07e-01$ & $7.07e-01$  & $7.07e-01$ & $7.06e-01$ & $7.01e-01$ \\
% 	\hline	
% 	\end{tabular}
% 	\caption {
% 	    Counterexample 1~\eqref{eqn:cex-1}: Compare the noDQ and DQ pointwise ROM error over shorter time interval $[0,0.05]$ and fixed time step $\Delta t = 0.01$ and varied $r$ values.
% 	} 
% \end{table} 

% \begin{table}[h!] 
% 	\centering
% 	\begin{tabular}{|c | c| c| c| c|  } 
% 		\hline
% 	$r$ & $ 1 $ & $2$  & $ 3 $ & $ 4 $   \\
% 		\hline\hline
% 	$||e_{noDQ}||_{\infty}$ & $1.14e+00 $ & $2.20e+00$ & $2.28e+00$ &  $2.29e+00$ \\
% 	\hline
% 	$||e_{DQ}||_{\infty}$ & $1.92e+00$ & $1.68e+00$ & $2.04e+00$ & $2.20e+00$  \\
% 	\hline	
% 	\end{tabular}
% 	\caption {
%     Counterexample 2~\eqref{eqn:cex-2}: Compare the noDQ and DQ pointwise ROM error over shorter time interval $[0,0.06]$ and fixed time step $\Delta t = 0.02$ and varied $r$ values.
% 	} 
% \end{table} 

\section{Conclusions}
    \label{sec:conclusions}

%\blue{
%Check section/theorem/equation numbers, which are hardwired for now.
%}

In this paper, we resolved several theoretical issues dealing with the optimality of pointwise in time error bounds for POD model order reduction of the heat equation.
In particular, we studied the role played by the DQs in the optimality of pointwise POD error bounds with respect to (i) the time discretization error, and (ii) the ROM discretization error.

First, in the noDQ case (i.e., when the DQs are not used to construct the POD basis), we proved that the error bound is suboptimal not only with respect to the ROM discretization (as shown in~\cite{iliescu2014are}), but also with respect to the time discretization.
Specifically, in Proposition \ref{prop:counterexample} we constructed two classes of analytical examples, and we proved that these examples violate Assumption \ref{ass:LinftyTime}, and yield suboptimal (with respect to the time discretization) pointwise projection error bounds.
Furthermore, we noted that these suboptimal pointwise projection error bounds yield suboptimal ROM error bounds (see Remark \ref{remark:noDQ_error_estimates}).
Finally, we illustrated the suboptimality of the pointwise projection and ROM error bounds in the numerical simulation of the heat equation.

%\smallskip

Our second main contribution is Theorem \ref{thm:uniform_estimates}, where we proved that,
%we proved in Theorem \ref{thm:uniform_estimates} that, 
in the DQ case (i.e., when the DQs are  used to construct the POD basis), Assumption \ref{ass:LinftyTime} is always satisfied.
To prove Theorem \ref{thm:uniform_estimates}, in Lemma \ref{lemma:discrete_time_Sobolev_embedding} we first proved a discrete time Sobolev inequality for the DQ case. Next, in Section \ref{sec:ErrEst}, we used Theorem \ref{thm:uniform_estimates} to prove pointwise ROM error bounds that are optimal with respect to both the ROM discretization error and the  time discretization error in the DQ case. 
In Section~\ref{sec:numerical-results}, we illustrated the optimality of the pointwise projection and error bounds in the numerical simulation of the heat equation.

%\smallskip

Our third main contribution is that, in Definition \ref{def:optimal}, we proposed a new  definition for the optimality of pointwise in time ROM discretization errors.
In Section \ref{sec:optimality}, we carefully discussed the relationship between this new optimality definition and the other two optimality definitions in current use.
In Theorem \ref{theorem:opt_equiv}, for two of the three optimality definitions, we showed that the DQ case yields optimal bounds, whereas the noDQ case yields suboptimal error bounds.  

%\bigskip

Our theoretical and numerical investigations (see also~\cite{iliescu2014are,KV01,singler2014new}) show that the DQs are needed to prove optimal pointwise in time error bounds.
There are, however, several research directions that need to be investigated.

%\smallskip

At a theoretical level, the uniform boundedness type conditions for non-orthogonal
POD projections considered in Proposition \ref{prop:compare_optimality_types} and Theorem \ref{theorem:opt_equiv} are important in proving some of the optimal pointwise ROM error bounds.
These type of uniform boundedness conditions have been studied both theoretically and numerically in~\cite{chapelle2012galerkin,iliescu2014are,KeanSchneier20,LockeSingler20,singler2014new,xie2018numerical}, but they are not well understood.
Further investigation of these conditions is needed.
%\medskip
Additionally, at a theoretical level we  considered optimal uniform estimates only for the heat equation. 
How these estimates will extend to more complicated nonlinear PDEs (e.g., the Navier-Stokes equations) is an open problem.
%\red{
In this paper, we considered equally spaced snapshots to construct the POD basis.
The POD adaptivity in time (see, e.g.,~\cite{alla2018posteriori,hoppe2014snapshot,kunisch2010optimal,oxberry2017limited} and the survey in~\cite{grassle2019adaptivity}) aims at choosing snapshot time instances that are optimal in some sense (e.g., such that the error between the ROM and FOM trajectories is minimized~\cite{kunisch2010optimal}). 
The effect of POD adaptivity in time on the optimality of error bounds in the noDQ and DQ cases should also be investigated.
%}
%\smallskip

At a numerical level, further investigation of the role of DQs in practical computations is needed.
The theoretical and numerical results in this paper focus exclusively on the optimality of the rates of convergence of ROM error bounds, but do not address the absolute size of the ROM error.
In our numerical investigation, the size of the ROM error was of the same order in the noDQ and DQ cases (results not included). 
In the current literature, the results do not yield a clear conclusion: 
In some references~\cite{homberg2003control,kostova2018model}, the ROM error is lower in the DQ case than in the noDQ case; in other references~\cite{iliescu2014are,KeanSchneier20,KV01}, the situation is reversed.
Further investigation of the relative size of the ROM error in the noDQ and DQ cases is needed.

%\newpage
\bibliographystyle{abbrv}
% \bibliographystyle{amsplain}
% \bibliography{Biblio_OPT-POD-ROM,traian,JRS}
\bibliography{references}

\end{document}